\documentclass[11pt]{article}

\usepackage{amstext,amssymb,amsmath,amsbsy,bm}
\usepackage{hyperref}
\usepackage{amscd}
\usepackage{amsfonts}
\usepackage{indentfirst}
\usepackage{verbatim}
\usepackage{amsmath}
\usepackage{amsthm}
\usepackage{enumerate}
\usepackage{graphicx}
\usepackage{color}
\usepackage[OT1]{fontenc}
\usepackage[latin1]{inputenc}
\usepackage[english]{babel}
\usepackage{amssymb}
\usepackage{subfig}
\usepackage{algorithm}
\usepackage{algpseudocode}
\usepackage[shortlabels]{enumitem}

\newcommand{\R}{\mathbb{R}}
\newcommand{\ds}{\displaystyle}

\newcommand{\x}{{\bf x}}

\newcommand{\bv}{{\bf v}}

\newcommand{\bn}{{\bf n}}
\newcommand{\bP}{{\bf P}}
\newcommand{\bmm}{{\bf m}}
\newcommand{\bN}{{\bf N}}

\newtheorem{Theorem}{Theorem}[section]
\newtheorem{Lemma}{Lemma}[section]

\newtheorem{Remark}{Remark}[section]

\newtheorem*{Assumption*}{Assumption}

\newtheorem{Problem}{Problem}[section]
\newtheorem*{Problem*}{Problem}
\numberwithin{equation}{section}

\usepackage{matlab-prettifier}

\begin{document}

\title{A Carleman-Picard approach for reconstructing zero-order coefficients in parabolic equations with limited data}

\author{
Ray Abney\thanks{Department of Mathematics and Statistics, University of North Carolina at
Charlotte, Charlotte, NC 28223, USA, \texttt{rgabney@charlotte.edu} }
\and
Thuy T. Le\thanks{%
Department of Mathematics, NC State University, Raleigh, NC 27695, USA, \texttt{tle9@ncsu.edu}}  \and Loc H. Nguyen\thanks{Department of Mathematics and Statistics, University of North Carolina at
Charlotte, Charlotte, NC, 28223, USA, \texttt{loc.nguyen@charlotte.edu}.} 
\and
Cam Peters\thanks{Department of Mathematics and Statistics, University of North Carolina at
Charlotte, Charlotte, NC, 28223, USA,  \texttt{cpeter68@charlotte.edu}.}
}

\date{}
\maketitle
\begin{abstract}
We propose a globally convergent computational technique for the nonlinear inverse problem of reconstructing the zero-order coefficient in a parabolic equation using partial boundary data. This technique is called the ``reduced dimensional method". Initially, we use the polynomial-exponential basis to approximate the inverse problem as a system of 1D nonlinear equations. We then employ a Picard iteration based on the quasi-reversibility method and a Carleman weight function. We will rigorously prove that the sequence derived from this iteration converges to the accurate solution for that 1D system without requesting a good initial guess of the true solution. The key tool for the proof is a Carleman estimate. We will also show some numerical examples.
\end{abstract}

\noindent{\it Keywords:
	time reduction,
	Carleman Picard iteration,
	nonlinear,
	parabolic.
}

\noindent{\it AMS subject classification: 	

}

\section{Introduction} \label{sec intr}

Let $d \geq 2$ be the spatial dimension. 
This paper aims to solve a coefficient inverse problem for the following initial value problem
\begin{equation}
	\left\{
		\begin{array}{rcll}
			u_t(\x, t) &=& \Delta u(\x, t) + c(\x)u(\x, t) &(\x, t) \in \R^d \times (0, \infty), \\
			u(\x, 0) &=& p(\x) &\x \in \R^d.
		\end{array}
	\right.
	\label{main_eqn}
\end{equation}
More precisely, we propose a globally convergent method to solve the following inverse problem.
\begin{Problem}
	Let $R$ and $T$ be two positive numbers.
	Define $\Omega = (-R, R)^{d}$, and	
	\begin{equation}
		\Gamma = \{
		 \x = (\x', z): \x' = (x_1, \dots, x_{d - 1}) \mbox{ and }, |x_i| < R, i = 1, \dots, d-1, z = \pm R\} 
		\subset \partial \Omega.
		\label{gamma}
	\end{equation}
	Assume that $|p| > 0$ in $\overline \Omega$.
	Given the boundary measurements
	\begin{equation}
		f(\x, t) = u(\x, t) 
		\mbox{ and }
		g(\x, t) = u_z(\x, t)
		\label{data}
	\end{equation}
	for all $(\x, t) \in \Gamma \times (0, T)$, compute the coefficient $c(\x),$ for $\x \in \Omega.$
\label{p}
\end{Problem}

Problem \ref{p} boasts countless real-world applications. Consider a scenario where the internal points of the medium $\Omega$ remain inaccessible. By recording partial boundary data of the function $u$, specifically the heat and heat flux as discussed in this paper, over a designated time frame and by resolving Problem \ref{p}, one can identify the coefficient $c(\x)$, $\x \in \Omega$. This allows the examination of the medium without causing any damage to it. An important example can be drawn from bioheat transfer, where the coefficient $c(\x)$ signifies blood perfusion. Understanding this coefficient is vital for determining the temperature of blood coursing through tissue, as highlighted in \cite{CaoLesnic:amm2019}.
However, the uniqueness of Problem \ref{p}, especially when data collection is limited to a specific subset of $\partial \Omega$, remains an open area and is explored within the reduced dimensional framework of this paper. Variations of Problem \ref{p}, with some internal data assumed to be known, have been addressed in \cite{BeilinaKlibanovBook, BukhgeimKlibanov:smd1981, PrilepkoKostin:RASBM1993}. Additionally, the uniqueness can be found in \cite{Isakov:ip1999} when provided with the Dirichlet to Neuman map. In this paper, the uniqueness of Problem \ref{p} is assumed.
Another topic of interest is the inverse challenge of retrieving other coefficients, such as diffusion or initial conditions, based on the final time measurements or boundary measurements for parabolic equations. This is an intriguing and critical issue, with theoretical findings and computational methods elaborated in \cite{AbhishekLeNguyenKhan, klibanovYagola:arxiv2019, LeCON2023, LeNguyen:jiip2022, LiNguyen:IPSE2020, TuanKhoaAu:SIAM2019, Prilepko:pam2000, Tuan:ip2017}.

Inverse problems of computing the coefficients for parabolic equations have been extensively explored. To the authors' knowledge, the widely-used technique for resolving such issues is the optimal control approach; see the important works \cite{Borceaetal:ip2014, CaoLesnic:nmpde2018, CaoLesnic:amm2019, KeungZou:ip1998, YangYuDeng:amm2008} and other cited references. The researchers in \cite{Borceaetal:ip2014} employed the optimal control method with a preconditioner to achieve high-quality numerical calculations of thermal conductivity. However, a significant limitation of this technique is the necessity for a reliable initial estimation of the true solution, which is not consistently accessible.
We would like to particularly highlight the convexification method, as described in \cite{KlibanovNik:ra2017, Klibanov:ip2015, KlibanovNguyenTran:JCP2022,  LeLeNguyen:Arxiv2022}. This approach addresses the challenge of obtaining an initial guess. The studies in\cite{KlibanovNik:ra2017, Klibanov:ip2015, KlibanovNguyenTran:JCP2022,  LeLeNguyen:Arxiv2022} suggested to minimize some Carleman weighted strictly convex functionals. When minimized, the minimizers of these functionals produce the solution to the problem at hand. 
Other worthy mentions are \cite{Nguyen:CAMWA2020} and \cite{Nguyens:jiip2020}, which respectively present alternative approaches to address Problem \ref{p} by iteratively solving a Picard-like approximation and its linearization. 
The approaches above consider the full boundary observation.
Unlike this, our contribution is introducing a fresh technique that does not rely on prior insights into the actual coefficient and requests only partial observation.

Our approach to addressing Problem \ref{p} is split into two phases.
In the initial phase, drawing inspiration from \cite{Nguyen:CAMWA2020, Nguyens:jiip2020}, we eliminate the unknown coefficient 
$c(\x)$ from \eqref{main_eqn}. By this, we obtain a partial differential equation. The equation that emerges from this phase is a complicated one, which involves both nonlocal and nonlinear terms. On the other hand, the boundary condition of the solution is only provided on $\Gamma \subsetneq \partial \Omega$. As of now, there is no established numerical method to address it.
During the subsequent phase, we transform this equation into a system of nonlinear ordinary differential equations. This transformation is guided by truncating the Fourier series with respect to a special basis introduced in paper \cite{NguyenLeNguyenKlibanov:2023}. This basis is named the polynomial-exponential basis.
It is the high-dimensional version of the 1D polynomial-exponential basis originally introduced in \cite{Klibanov:jiip2017}. We then deploy a predictor-corrector strategy to solve this nonlinear system. 
Within this framework, the preliminary approximation of the true solution is derived without any prior understanding. Subsequently, the resolution to Problem \ref{p} is achieved.
The corrector stage in this procedure is executed using the quasi-reversibility method and a Carleman weight function.
The quasi-reversibility method was first introduced by Latt\`es and Lions in \cite{LattesLions:e1969} for numerical solutions of ill-posed problems for
partial differential equations. 
It has been studied intensively since then,
see e.g., \cite{Becacheelal:AIMS2015, Bourgeois:ip2006,
BourgeoisDarde:ip2010, BourgeoisPonomarevDarde:ipi2019, ClasonKlibanov:sjsc2007, Dadre:ipi2016, KaltenbacherRundell:ipi2019,
KlibanovSantosa:SIAMJAM1991, Klibanov:jiipp2013, LocNguyen:ip2019, NguyenLiKlibanov:2019}. A survey
on this method can be found in \cite{Klibanov:anm2015}. 
A question arises immediately whether or not the iteration led by the predictor-corrector procedure above converges.
In this paper, we will rigorously prove this important result.
The proof is motivated by the one in \cite{LeCON2023, LeNguyen:jiip2022, Nguyen:AVM2023}. However, its advantage is that we can relax a technical condition in those papers about the smoothness of the noise.
That means the noise model in this paper is more realistic than in the earlier publications.

The paper is organized as follows.
In Section \ref{sec2}, we introduce our approximation dimensional model that leads to the dimensional reduction approach.
In Section \ref{sec3}, we establish a 1D Carleman estimate.
Section \ref{sec4} is for the algorithm and the proof of its convergence.
In Section \ref{sec5}, we present our numerical study.
Section \ref{sec6} is for concluding remarks.

\section{The reduced dimension model} \label{sec2}

Define 
\begin{equation}
	v(\x, t) = u_t(\x, t)
	\quad 
	\mbox{for all }
	(\x, t) \in \Omega \times (0, T).
	\label{changev}
\end{equation}
Then, by differentiating both sides of the differential equation in \eqref{main_eqn} with respect to $t$, we obtain
\begin{equation}
	v_t(\x, t) = \Delta v(\x, t) + c(\x) v(\x, t)
	\quad 
	\mbox{for all }
	(\x, t) \in \Omega \times (0, T).
	\label{2.2}
\end{equation}
Since $u(\x, 0) = p(\x)$, we have 
\begin{equation}
	v(\x, 0) = 	u_t(\x, 0) = \Delta u(\x, 0) +  c(\x) u(\x, 0)
	=  \Delta p(\x) +  c(\x) p(\x)
	\quad 
	\mbox{for all }
	\x \in \Omega.
	\label{2..3}
\end{equation}
Recall the assumption that $|p(\x)| > 0$ for $\x \in \overline \Omega$. Due to \eqref{2..3}
 \begin{equation}
	c(\x) = \frac{v(\x, 0) - \Delta p(\x)}{p(\x)}
	\quad 
	\mbox{for all }
	\x \in \Omega.
	\label{2.3}
\end{equation}
Plugging $c(\x)$, computed in \eqref{2.3}, into \eqref{2.2} gives
\begin{equation}
	v_t(\x, t) = \Delta v(\x, t) + \frac{v(\x, 0) - \Delta p(\x)}{p(\x)} v(\x, t)
	\quad 
	\mbox{for all }
	(\x, t) \in \Omega \times (0, T).
	\label{2.4}
\end{equation}
Equation \eqref{2.4} is nonlinear and nonlocal.
A theory to solve it is not yet available.
We propose the following dimensional reduction approach to solve it.

\begin{Remark}
	The change of variable in \eqref{changev} and the elimination of $c$ to derive equation \eqref{2.4} were adopted in \cite{Nguyen:CAMWA2020, Nguyens:jiip2020}.
\end{Remark}

For each $n \in \mathbb{N}$, define $\phi_n(t) = t^{n - 1} e^t$ for all $t \in (0, T)$ and $\Phi_n(x) = x^{n - 1}e^x$ for all $x \in (-R, R)$.
The sets $\{\phi_n\}_{n \geq 1}$ and $\{\Phi_n\}_{n \geq 1}$ are complete in $L^2(0, T)$ and $L^2(-R, R)$ respectively.
Applying the Gram-Schmidt orthonormalization process on these two sets, we obtain orthonormal bases $\{\psi_n\}$ and $\{\Psi_n\}$ of $L^2(0, T)$ and $L^2(-R, R)$ respectively.
For each multi-index $\bn = (n_1, \dots , n_{d - 1}, n_t) \in \mathbb{N}^{d}$, define the $d$-dimensional tensor-valued function ${\bf P}_{\bn}$ as
\[
	{\bf P}_{\bf n}(\x', t) 
	= \Psi_{n_1}(x_1) \dots \Psi_{n_{d - 1}}(x_{d - 1})\psi_{n_t}(t) 
\]
for all $\x' = (x_1, \dots, x_{d - 1}) \in (-R, R)^{d-1}, t \in (0, T)$.
It is obvious that the set $\{{\bf P}_{\bf n}: \bf n \in \mathbb{N}^d\}$ is an orthonormal basis of the space $L^2(\Omega' \times (0, T))$.
We name this basis the polynomial-exponential basis.
The 1D version of the polynomial-exponential basis was introduced in \cite{Klibanov:jiip2017}, and the higher dimension version was defined in \cite{NguyenLeNguyenKlibanov:2023}.
\begin{Remark}
The one-dimensional version of the polynomial-exponential basis was employed to solve a list of inverse problems, including  nonlinear inverse problems with simulated and experimental data \cite{VoKlibanovNguyen:IP2020,
Khoaelal:IPSE2021, KhoaKlibanovLoc:SIAMImaging2020, ThuyKhoaKlibanovLocBidneyAstratov:2023,
LeNguyen:JSC2022, Nguyen:CAMWA2020, Nguyens:jiip2020} and linear inverse problems \cite{LeNguyenNguyenPowell:JOSC2021, NguyenKlibanov:ip2022, Nguyens:jiip2020, KlibanovAlexeyNguyen:SISC2019}.
\end{Remark}

From now on, for all $\x \in \Omega$ we write $\x = (\x', z) \in \Omega' \times \R$ where $\x' \in \Omega':= (-R, R)^{d - 1}$ consists of the first $d - 1$ coordinates and $z \in (-R, R)$ is the last coordinate of $\x$.
Then, by expanding the function $v(\x', z, t)$ using the basis $\big\{\bP_{\bn}: \bn \in \mathbb{N}^d\big\}$, we can approximate the function $v(\x', z, t)$ as follows
\begin{align}
	v(\x', z, t) &= \sum_{\bn \in \mathbb{N}^d} v_{\bn}(z) \bP_{\bn}(\x', t)	
	\simeq
	\sum_{\bn \leqq \bN} v_{\bn}(z) \bP_{\bn}(\x', t)
	\nonumber
	\\
	&=
	\sum_{n_1 = 1}^{N_1}
	\dots
	\sum_{n_{d - 1}}^{N_{d - 1}}
	\sum_{n_t = 1}^{N_t}
	v_{(n_1, \dots, n_{d - 1}, n_t)} (z)
	\Psi_{n_1}(x_1) \dots \Psi_{n_{d - 1}}(x_{d - 1})\psi_{n_t}(t)
	\label{2.5}
\end{align}
for $(\x', z, t) \in \Omega' \times (-R, R) \times (0, T)$, 
where $\bN = (N_1, \dots, N_{d - 1}, N_t)$ represents a cut-off vectors
and
\begin{equation}
	v_{\bn}(z) = \int_{\Omega' \times (0, T)} v(\x', z, t) \bP_{\bn}(\x', t)d\x' dt.
	\label{vn}
\end{equation}
The values of the cut-off numbers $N_1, \dots, N_{d- 1},$ and $N_t$ will be chosen based on the given data in \eqref{data}.
See Section \ref{sec52} and Figure \ref{ChooseN} for an illustration of a suitable choice of these numbers.
In \eqref{2.5} and hereafter, we understand $\bn \leqq \bN$ by the statement that
\begin{multline}
	\bn = (n_1, \dots, n_{d- 1}, n_t) \leqq \bN = (N_1, \dots, N_{d- 1}, N_t)
	\mbox{ is equivalent to }
	\\
	1 \leq n_1 \leq N_1, \dots, 1 \leq n_{d - 1} \leq N_{d-1}, \mbox{ and }
	1 \leq n_t \leq N_t.
	\label{order}
\end{multline}

We assume that the approximation \eqref{2.5} is valid. Plugging \eqref{2.5} into \eqref{2.4} gives
\begin{multline}
	\sum_{\bn \leqq \bN} v_{\bn}(z) \frac{\partial \bP_{\bn}(\x', t)}{\partial t}
 =
 	\sum_{\bn \leqq \bN} v_{\bn}''(z) \bP_{\bn}(\x', t)
	+
	\sum_{\bn \leqq \bN} v_{\bn}(z) \Delta_{\x'}\bP_{\bn}(\x', t)
	\\
		+
		\frac{1}{p(\x)}\Big(\sum_{\bn \leqq \bN} v_{\bn}(z) \bP_{\bn}(\x', 0) - \Delta p(\x)\Big) \sum_{\bn \leqq \bN} v_{\bn}(z) \bP_{\bn}(\x', t)
\label{2.6}
\end{multline}
for all $(\x', z, t) \in \Omega' \times (-R, R) \times (0, T).$
For each multi-index $\bmm \leqq \bN$, we multiply $\bP_{\bmm}(\x', t)$ to both sides of \eqref{2.6}, and then integrate the resulting equation over $\Omega' \times (0, T)$ to get
\begin{multline}
	\sum_{\bn \leqq \bN} v_{\bn}(z) \int_{\Omega' \times (0, T)} \frac{\partial \bP_{\bn}(\x', t)}{\partial t} \bP_{\bmm}(\x', t)d\x'dt
	\\
 =
 	\sum_{\bn \leqq \bN} v_{\bn}''(z) \int_{\Omega' \times (0, T)} \bP_{\bn}(\x', t) \bP_{\bmm}(\x', t)d\x'dt
	+
	\sum_{\bn \leqq \bN} v_{\bn}(z)\int_{\Omega' \times (0, T)} \Delta_{\x'}\bP_{\bn}(\x', t) \bP_{\bmm}(\x', t)d\x' dt
	\\
		+
	\frac{1}{p(\x)}\Big(\sum_{\bn \leqq \bN} v_{\bn}(z) \bP_{\bn}(\x', 0) - \Delta p(\x)\Big) 
	\int_{\Omega' \times (0, T)}\sum_{\bn \leqq \bN} v_{\bn}(z) \bP_{\bn}(\x', t) bP_{\bmm}(\x', t)d\x'dt
\label{2.7}
\end{multline}
for all $z \in (-R, R).$
Defining
\begin{equation*}
	s_{\bmm \bn}
	= -\int_{\Omega' \times (0, T)} \frac{\partial \bP_{\bn}(\x', t)}{\partial t} \bP_{\bmm}(\x', t)d\x'dt + \int_{\Omega' \times (0, T)} \Delta_{\x'}\bP_{\bn}(\x', t) \bP_{\bmm}(\x', t)d\x' dt
\end{equation*}
and
\begin{equation}
 {\bf F_{\bmm}}([v_{\bn}(z)]_{\bn \leqq \bN} )
	= \frac{1}{p(\x)}\Big(\sum_{\bn \leqq \bN} v_{\bn}(z) \bP_{\bn}(\x', 0) - \Delta p(\x)\Big) 
 v_{\bmm}(z),
 \label{Fm}
\end{equation}
 we obtain from \eqref{2.7} that
 \begin{equation}
 	v_{\bmm}''(z) + \sum_{\bn \leqq \bN} s_{\bmm \bn} v_{\bn}(z) +  {\bf F_{\bmm}}([v_{\bn}(z)]_{\bn \leqq \bN} ) = 0
	\label{2.8}
 \end{equation}
for all $z \in (-R, R)$ and for all $\bmm \leqq \bN$, see \eqref{order} for the definition of $\leqq$.
Coupling all equations \eqref{2.8} for $\bmm \leqq \bN$ forms a system second-order ordinary equations for the $d$-dimensional valued tensor ${\bf v}(z) = [v_{\bn}(z)]_{\bn \leqq \bN},$ $z \in (-R, R)$.
The Cauchy boundary conditions for the tensor ${\bf v}$ can be derived from \eqref{data} and \eqref{vn}, read as
\begin{equation}
	{\bf v}(\pm R) = [v_{\bn}(\pm R)]_{\bn \leqq \bN}
	= 
	\Big[\int_{\Omega' \times (0, T)} f_t(\x',\pm R,t) \bP_{\bn}(\x', t)dt\Big]_{\bn \leqq \bN}
	\label{2.10}
\end{equation}
and
\begin{equation}
	{\bf v}'(\pm R) = [ v_{\bn}'(\pm R)]_{\bn \leqq \bN}
	= 
	\Big[\int_{\Omega' \times (0, T)} g_t(\x',\pm R,t) \bP_{\bn}(\x',t)dt\Big]_{\bn \leqq \bN}.
	\label{2.11}
\end{equation}
Combining \eqref{2.8}, \eqref{2.10}, and \eqref{2.11}, we obtain a system of Cauchy problem for ${\bf v} = [v_{\bn}(z)]_{\bn \leqq \bN}$
\begin{equation}
	\left\{
	\begin{array}{ll}
		v_{\bmm}''(z) + \sum_{\bn \leqq \bN} s_{\bmm \bn} v_{\bn}(z) +  {\bf F_{\bmm}}({\bf v} ) = 0
		&z \in (-R, R),
		\\
		v_{\bmm}(z) 
	= P_{\bmm}(z)
	\ds &z = \pm R,
	\\
		v_{\bmm}'(z)
	= 
	Q_{\bmm}(z) &z = \pm R,
	\end{array}
	\right.
	\quad \mbox{for } \bmm \leqq \bN.
\label{2.1212}
\end{equation}
where 
\begin{align}
	P_{\bmm}(z) &= \int_{\Omega' \times (0, T)} f_t(\x', z, t) \bP_{\bmm}(\x', t)dt, 
	\label{2.1414}
	\\
	Q_{\bmm}(z) &= \ds\int_{\Omega' \times (0, T)} g_t(\x', z, t) \bP_{\bmm}(\x',t)dt.
	\label{2.15}
\end{align}
Introduce the ``tensor multiplication" operator
\begin{equation*}
	\mathcal{S} :: \bv = \big[\sum_{\bn \leqq \bN} s_{\bmm \bn} v_{\bn}\big]_{\bmm \leqq \bN}
\end{equation*}
and 
the notations
\begin{align*}
	\mathcal{F}(\bv) &= [{\bf F}_{\bmm}(\bv)]_{\bmm \leqq \bN},\\
	\mathcal{P}(z) &= [P_{\bmm}(z)]_{\bmm \leqq \bN}, \quad z = \pm R,\\
	\mathcal{Q}(z) &= [Q_{\bmm}(z)]_{\bmm \leqq \bN}, \quad z = \pm R.
\end{align*}
We shorten the coupling system in \eqref{2.1212} as
\begin{equation}
	\left\{
		\begin{array}{ll}
			\bv''(z) + \mathcal{S}::\bv(z) + \mathcal{F}(\bv(z)) = 0 &z \in (-R, R),\\
			\bv(z) = \mathcal{P}(z) &z = \pm R,\\
			\bv'(z) = \mathcal{Q}(z) &z = \pm R.
		\end{array}
	\right.
	\label{2.12}
\end{equation}

\begin{Remark}
	Computing the values of $\bv$ and $\bv'$ at $z = \pm R$ in \eqref{2.10}, \eqref{2.11}, and \eqref{2.12} requires us to differentiate the given data $f$ and $g$ with respect to the time $t$.
	This task is not trivial, especially when the data are corrupted by noise.
	In this paper, we employ the new differentiating technique in \cite{NguyenLeNguyenKlibanov:2023}, in which we approximate the data by eliminating their high-frequency terms from the Fourier expansion of the given data with respect to the polynomial-exponential basis before differentiating.
	It was numerically shown in \cite{NguyenLeNguyenKlibanov:2023} that computing derivatives using this new technique is more accurate than the conventional ones; say the finite difference, the cubic spline, and the Tikhonov optimization methods.
\end{Remark}

\begin{Remark}
	The first key point of our dimension reduction approach lies in the derivation of the approximation model \eqref{2.12}, a system of first-order ODEs along the $z-$axis.	
	The approximation model \eqref{2.12} involves 
	\begin{equation}
		|\bN| = |(N_1, \dots, N_{d-1}, N_t)|=  N_t\prod_{i = 1}^{d - 1} N_{i}
		\label{bN}
	\end{equation} equations versus the same numbers of unknown entries of ${\bf v} = [v_{\bmm}]_{\bmm \leqq \bN}$.
	 This allows for the computation of the tensor-valued function $\bv(z)$ for $z \in (-R, R)$, and subsequently the function $v(\x', z, t)$ for all $(\x', z, t) \in \Omega' \times (-R, R) \times (0, T)$.
	The solution $c(\x),$ $\x \in \Omega$, to Problem \ref{p} can be computed via the knowledge of $v$ and the reconstruction formula \eqref{2.3}.
	However, this convenience comes with a trade-off.
	The truncation in \eqref{2.4} makes system \eqref{2.12}  not exact. 
	It should be considered as an approximation context for Problem \ref{p}.
	Studying the behavior of \eqref{2.12} when all cut-off numbers $N_2, \dots, N_d, N_t$ tend towards $\infty$ presents a significant challenge. This paper does not cover this complex topic, which prioritizes computational aspects. In exchange, we will show that our dimension reduction method is acceptable in numerics.  It can quickly deliver reliable solutions since we have transferred a high dimensional problem into a problem along the $z-$axis, which is a 1D problem.	
	\label{rem22}
\end{Remark}

As noted in Remark \ref{rem22}, once the system of ODEs in \eqref{2.12} with Cauchy boundary data is solved, the computed solution to Problem \ref{p} follows.
However, this task is challenging since \eqref{2.12} is nonlinear. 
There are several methods to solve nonlinear systems of ODEs.
The conventional approach is based on optimization.
For example, one can solve \eqref{2.12} by minimizing the least squares cost functional
\begin{equation}
	J_{\rm lsq}({\bf v})
	= \int_{-R}^R \big|\bv''(z) + \mathcal{S}::\bv +  \mathcal{F}(\bv)\big|^2 dz + \mbox{a regularization term}
	\label{2.13}
\end{equation}
subject to the endpoint condition in \eqref{2.12} and then accepting the minimizer as the computed solution.
This method is effective when a good initial guess of \eqref{2.12} is given because $J_{\rm lsq}$ might have multiple local minima.  The challenge is that such an initial guess is not always available in practical applications. 
Consequently, the optimization approach is not deemed suitable for solving \eqref{2.12}.
There are three approaches to solve \eqref{2.12} without requesting a good initial guess, all based on Carleman convexification.
\begin{enumerate}
	\item {\it The Carleman convexification method}. The key of the Carleman convexification method is to include a Carleman weight function; e.g., $W_\lambda(z) = e^{-\lambda z}$, $\lambda > 1$,  to the least squares cost functional in \eqref{2.13}. That means one can minimize the Carleman weighted functional
	\begin{equation*}
		J_{\rm conv}({\bf v})
	= \int_{-R}^R W_\lambda(z) \big|\bv''(z) + \mathcal{S}::\bv +  \mathcal{F}(\bv)\big|^2 dz + \mbox{a regularization term},
	\end{equation*}
	subject to the boundary conditions in \eqref{2.12} where ${\bf v} = [v_{\bmm}]_{\bmm \leqq \bN}$.
	One can prove that $J_{\rm conv}$ is uniformly convex in any bounded subset of the functional space containing the desired solution provided that $\lambda$ is sufficiently large.
	Also, the unique minimizer is close to the true solution to \eqref{2.12}. 
	 The original convexification method was first introduced in \cite{KlibanovIoussoupova:SMA1995}, with subsequent results found in \cite{KlibanovNik:ra2017, KhoaKlibanovLoc:SIAMImaging2020, KlibanovNguyenTran:JCP2022, LeNguyen:JSC2022}. Despite its efficacy in producing reliable numerical solutions, the convexification method has a high computational cost.
	\item {\it The Carleman contraction method}.
	The contraction method for solving \eqref{2.12} primarily starts with an initial function ${\bf v}^{(0)} = [v_{\bmm}^{(0)}]_{\bmm \leqq \bN}$. 
	Note that ${\bf v}^{(0)}$ might be far away from the true solution to \eqref{2.12}. From this point, given that ${\bf v}^{(k)} = [v_{\bmm}^{(k)}]_{\bmm \leqq \bN}$, $k \geq 0$, is known, we compute ${\bf v}^{(k + 1)} = [v_{\bmm}^{(k + 1)}]_{\bmm \leqq \bN}$ as the ``Carleman-regularized" solution to
	\begin{equation}
	\left\{
		\begin{array}{ll}
			{\bv^{(k+1)}}''(z) + \mathcal{S}::\bv^{(k + 1)}(z) + \mathcal{F}(\bv^{(k)}(z)) = 0 &z \in (-R, R),\\
			\bv(z) = \mathcal{P}(z) &z = \pm R,\\
			\bv'(z) = \mathcal{Q}(z) &z = \pm R.
		\end{array}
	\right.
	\label{2.14}
\end{equation}
By Carleman-regularized solution, we mean ${\bf v}^{(k+1)}$ is the minimizer of
\begin{equation*}
	J^{(k)}(\bv)
	= \int_{-R}^R W_\lambda(z) \big|\bv''(z) + \mathcal{S}::\bv +  \mathcal{F}(\bv^{(k)})])\big|^2 dz
	 + \mbox{a regularization term}
\end{equation*}
subject to the boundary conditions in \eqref{2.12}.
The choice of $\lambda$, $W_\lambda(z)$, and the regularization term will be specified later.
The procedure to compute ${\bf v}^{(k+1)}$ above  involves the combination of the quasi-reversibility method \cite{LattesLions:e1969} and an appropriate Carleman estimated, as in \cite{LeCON2023, LeNguyen:jiip2022,  Nguyen:AVM2023}.
   Thanks to the presence of the Carleman weight function $W_{\lambda}(z)$, one can follow the arguments in \cite{LeCON2023, LeNguyen:jiip2022, Nguyen:AVM2023} to prove the convergence of the constructed sequence $\big\{{\bf v}^{(k)}\big\}_{k \geq 0}$ to the true solution to \eqref{2.12}. 
    	\item {\it The Carleman-Newton method. } The Carleman-Newton method is similar to the Carleman contraction method. Given an initial solution ${\bf v}^{(0)}$ that can be chosen arbitrary, we find ${\bf v}^{(1)}$ as the Carleman regularized  solution to the linearization of \eqref{2.12} about ${\bf v}^{(0)}$. 
	We refer the reader to \cite{AbhishekLeNguyenKhan, LeNguyenTran:CAMWA2022} for details and the rigorous proof of the convergence due to the Carleman-Newton method.
\end{enumerate}

Among the three methods mentioned above, we will choose the second approach; i.e., we will establish a 1D analog of the Carleman contraction method to solve \eqref{2.12}. This choice is appropriate due to the global convergence, the rapid rate of convergence, and the simplicity of the computational implementation.
In the previous two sentences, we mentioned ``analog" because $\mathcal{F}$ does not satisfy the Lipschitz condition in \cite{Nguyen:AVM2023}, which requires some modification in analysis.

In the next section, we establish a Carleman estimate, which plays an important role in proving the convergence of the Carleman contraction method.

\section{A 1D-Carleman estimate}\label{sec3}

Let $z_0 < -R$ be a fixed number.
We have the lemma.
\begin{Lemma}
There is a number $\lambda_0 > 1$ and a constant $C > 0$ depending only on $R$ and $z_0$ such that
	for all function $w \in C^2([-R, R])$, we have
	\begin{multline}
		\int_{-R}^Re^{2\lambda (z - z_0)^{-2}} |w''(z)|^2 dz
		\geq  
		-Ce^{2\lambda (R-z_0)^{-2}}(\lambda^3 |w(R)|^2 + \lambda |w'(R)|^2) 
		\\
		- C e^{2\lambda (-R-z_0)^{-2}} (\lambda^3 |w(-R)|^2 + \lambda |w'(-R)|^2)
		+C \lambda^3 \int_{-R}^R e^{2\lambda (z - z_0)^{-2}}|w(z)|^2dz		
		\\
		+C \lambda \int_{-R}^R e^{2\lambda (z - z_0)^{-2}} |w'(z)|^2dz.
		\label{Car}
	\end{multline}
	\label{lem31}
\end{Lemma}
\begin{proof}
	{\it Step 1.} Define 
	\begin{equation}
		y(z) = e^{\lambda (z - z_0)^{-2}} w(z)
		\quad
		\mbox{ or equivalently }
		\quad
		w(z) = e^{-\lambda (z - z_0)^{-2}} y(z)
		\label{3.2}
	\end{equation}
	for all $z \in [-R, R].$
	We have
	\begin{equation}
		w'(z) = e^{-\lambda (z - z_0)^{-2}}\Big[
			2\lambda (z - z_0)^{-3} y(z) + y'(z)
		\Big]
	\end{equation}	
	and
	\begin{equation}
		w''(z) = e^{-\lambda (z - z_0)^{-2}}\Big[
			2\lambda(z - z_0)^{-6}[3(z - z_0)^2 - 2 \lambda] y(z)		
		+ 4\lambda (z - z_0)^{-3} y'(z)
		+ y''(z)
		\Big]
	\end{equation}	
	for all $z \in [-R, R]$.
	Thus,
	\begin{align*}
		e^{2\lambda (z - z_0)^{-2}} |w''(z)|^2
		&= \Big[
			2\lambda(z - z_0)^{-6}[3(z - z_0)^2 - 2 \lambda] y(z)		
		+ 4\lambda (z - z_0)^{-3} y'(z)
		+ y''(z)
		\Big]^2
		\\
		&\geq 16\lambda^2(z - z_0)^{-9}[3(z - z_0)^2 - 2 \lambda] y(z)  y'(z)
		+ 
		8\lambda (z - z_0)^{-3} y'(z) y''(z)
		\\
		&=8\lambda^2(z - z_0)^{-9}[3(z - z_0)^2 - 2 \lambda] \frac{d}{dz}|y(z)|^2
		+ 4\lambda (z - z_0)^{-3} \frac{d}{dz}|y'(z)|^2
	\end{align*}
	for all $z \in [-R, R].$ 
	Here, we have used the inequality $(a + b + c)^2 \geq 2 ab + 2 bc.$
	Thus,
	\begin{equation}
		(z - z_0)^{10}e^{2\lambda (z - z_0)^{-2}} |w''(z)|^2 
		\geq 8\lambda^2 (z - z_0)[3(z - z_0)^2 - 2\lambda] \frac{d}{dz}|y(z)|^2
		+ 4\lambda (z - z_0)^7\frac{d}{dz}|y'(z)|^2
		\label{3.5}
	\end{equation}
	for all $z \in [-R, R].$ 
	By the product rule in differentiation $ab' = (ab)' - a'b$,
	we have
	\begin{multline}
		(z - z_0)^{10}e^{2\lambda (z - z_0)^{-2}} |w''(z)|^2 
		\geq  \frac{d}{dz} \Big[8\lambda^2 (z - z_0)[3(z - z_0)^2 - 2\lambda] |y(z)|^2
		\Big]
		\\
		-
		8\lambda^2  |y(z)|^2  \frac{d}{dz} \Big[(z - z_0)[3(z - z_0)^2 - 2\lambda]\Big]
		+ \frac{d}{dz} \Big[4\lambda (z - z_0)^7|y'(z)|^2\Big]
		\\
		- |y'(z)|^2\frac{d}{dz} \Big[4\lambda (z - z_0)^7\Big]
		\label{3.6}
	\end{multline}
	for all $z \in [-R, R].$ 
	Rearranging terms in \eqref{3.6} and simplifying the resulting inequality, we get
	\begin{multline*}
		(z - z_0)^{10}e^{2\lambda (z - z_0)^{-2}} |w''(z)|^2 
		\geq  \frac{d}{dz} \Big[8\lambda^2 (z - z_0)[3(z - z_0)^2 - 2\lambda] |y(z)|^2
		+ \frac{d}{dz} \Big[4\lambda (z - z_0)^7|y'(z)|^2\Big]
		\Big]
		\\
		-
		8\lambda^2  |y(z)|^2   \Big[9(z - z_0)^2 - 2\lambda\Big]		
		- 28|y'(z)|^2 \Big[\lambda (z - z_0)^6\Big]
%		\label{3.7}
	\end{multline*}
	for all $z \in [-R, R].$ 
	Therefore,
%	\begin{multline}
%		(z - z_0)^{10}e^{2\lambda (z - z_0)^{-2}} |w''(z)|^2 
%		\geq  \frac{d}{dz} \Big[8\lambda^2 (z - z_0)[3(z - z_0)^2 - 2\lambda] |y(z)|^2
%		\Big]
%		\\
%		-
%		8\lambda^2  |y(z)|^2  \frac{d}{dz} \Big[(z - z_0)[3(z - z_0)^2 - 2\lambda]\Big]
%		+ \frac{d}{dz} \Big[4\lambda (z - z_0)^7|y'(z)|^2\Big]
%		\\
%		- |y'(z)|^2\frac{d}{dz} \Big[4\lambda (z - z_0)^7\Big]
%		\label{3.6}
%	\end{multline}
%	for all $z \in [-R, R].$ 
%	Rearranging terms in \eqref{3.6} and simplifying the resulting inequality, we get
	\begin{multline}
		(z - z_0)^{10}e^{2\lambda (z - z_0)^{-2}} |w''(z)|^2 
		\geq  \frac{d}{dz} \Big[8\lambda^2 (z - z_0)[3(z - z_0)^2 - 2\lambda] |y(z)|^2
		+  \Big[4\lambda (z - z_0)^7|y'(z)|^2\Big]
		\Big]
		\\
		+16 \lambda^3 |y(z)|^2		-72\lambda^2(z - z_0)^2 |y(z)|^2
		- 28 \lambda |y'(z)|^2 (z - z_0)^6
		\label{3.7}
	\end{multline}
	for all $z \in [-R, R].$ 
	Integrating \eqref{3.7} over $[-R, R]$ and noting that $\lambda^3 \gg \lambda^2 \gg \lambda$ as $\lambda$ large, we can find a number $\lambda_0 > 1$ and a generic constant $C > 0$, both of which depend only on $z_0$ and $R$, such that
	\begin{multline}
		\int_{-R}^Re^{2\lambda (z - z_0)^{-2}} |w''(z)|^2 dz
		\geq  
		-C(\lambda^3 |y(R)|^2 + \lambda |y'(R)|^2) - C(\lambda^3 |y(-R)|^2 + \lambda |y'(-R)|^2)
		\\
		+C \lambda^3 \int_{-R}^R |y(z)|^2dz		
		-C \lambda \int_{-R}^R |y'(z)|^2dz.
		\label{3.9}
	\end{multline}
for all $\lambda \geq \lambda_0.$	
	
	{\it Step 2.}
	Recall from \eqref{3.2}
	that $y(z) = e^{\lambda (z - z_0)^{-2}} w(z).$ We have
	\[
		y'(z) = e^{\lambda (z - z_0)^{-2}}[-2\lambda (z - z_0)^{-3} w(z) + w'(z)].
	\]
	Thus, by the inequality $-(a + b)^2 \geq -2(a^2 + b^2)$, we have
	\begin{equation}
		-|y'(z)|^2 
		\geq -2e^{2\lambda (z - z_0)^{-2}}
		\Big[
			4\lambda^2(z - z_0)^{-6}|w(z)|^2 + |w'(z)|^2
		\Big]
		\label{3.10}
	\end{equation}
	for all $z \in [-R, R].$
	Combining \eqref{3.9} and \eqref{3.10} and recalling that $C$ is a generic constant depending only on $z_0$ and $R$,
	we have
	\begin{multline}
		\int_{-R}^Re^{2\lambda (z - z_0)^{-2}} |w''(z)|^2 dz
		\geq  
		-Ce^{2\lambda (R-z_0)^{-2}}(\lambda^3 |w(R)|^2 + \lambda |w'(R)|^2) 
		\\
		- C e^{2\lambda (-R-z_0)^{-2}} (\lambda^3 |w(-R)|^2 + \lambda |w'(-R)|^2)
		\\
		+C \lambda^3 \int_{-R}^R e^{2\lambda (z - z_0)^{-2}}|w(z)|^2dz		
		-C \lambda \int_{-R}^R e^{2\lambda (z - z_0)^{-2}} |w'(z)|^2dz.
		\label{3.11}
	\end{multline}
	
	{\it Step 3.}
	Using the inequality $2ab \leq a^2 + b^2$, we have
	\begin{align*}
	\int_{-R}^R e^{2\lambda (z - z_0)^{-2}} \frac{d}{dz}|w'(z)|^2dz 
	&= 
		2\int_{-R}^R e^{2\lambda (z - z_0)^{-2}} w'(z) w''(z)dz
		\\
		&\leq
		\int_{-R}^R
		e^{2\lambda (z - z_0)^{-2}} |w'(z)|^2dz 
		+ 
		\int_{-R}^R
		e^{2\lambda (z - z_0)^{-2}} |w''(z)|^2dz.
	\end{align*}
	Therefore,
	\begin{align*}
		\int_{-R}^R
		e^{2\lambda (z - z_0)^{-2}} |w''(z)|^2dz
		&\geq 
		\int_{-R}^R e^{2\lambda (z - z_0)^{-2}} \frac{d}{dz}|w'(z)|^2dz 
		- \int_{-R}^R e^{2\lambda (z - z_0)^{-2}} |w'(z)|^2dz
		\\
		&=\int_{-R}^R \frac{d}{dz}\Big[ e^{2\lambda (z - z_0)^{-2}} |w'(z)|^2\Big]dz
		- 
		\int_{-R}^R |w'(z)|^2 \frac{d}{dz} e^{2\lambda (z - z_0)^{-2}} dz
		\\
		&\hspace{6.5cm}
		- \int_{-R}^R e^{2\lambda (z - z_0)^{-2}} |w'(z)|^2dz.
	\end{align*}
	As a result,
	\begin{multline}
		\int_{-R}^R
		e^{2\lambda (z - z_0)^{-2}} |w''(z)|^2dz
		\geq 
		-C \Big[
			e^{2\lambda (R - z_0)^{-2}} |w'(R)|^2 + e^{2\lambda (-R - z_0)^{-2}} |w'(-R)|^2
		\Big]
		\\
		+ 
		2\lambda \int_{-R}^R (z - z_0)^{-3} e^{2\lambda (z - z_0)^{-2}} |w'(z)|^2   dz
		- \int_{-R}^R e^{2\lambda (z - z_0)^{-2}} |w'(z)|^2dz.
		\label{3.12}
	\end{multline}
	Adding \eqref{3.11} and \eqref{3.12} and recalling that $\lambda \geq \lambda_0 \gg 1$, 
	we obtain \eqref{Car}.
	
\end{proof}

\section{A Picard-like iteration to solve \eqref{2.12}}\label{sec4}

In this section, we employ the Carleman estimate in Lemma \ref{lem31} to construct a sequence that converges to the solution to \eqref{2.12}, provided that this true solution exists.
We consider the circumstance that the boundary data $\mathcal{P}$ and $\mathcal{Q}$ of \eqref{2.12} contain noise. 
Let $\mathcal{P}^*$ and $\mathcal{Q}^*$ be the unknown exact values of the boundary data $\mathcal{P}$ and $\mathcal{Q}$, respectively.
Let 	$\bv^*$ be the solution to \eqref{2.12} with $\mathcal{P}$ and $\mathcal{Q}$ being replaced by $\mathcal{P}^*$ and $\mathcal{Q}^*$, respectively. That means, $\bv^*$ solves
\begin{equation}
	\left\{
		\begin{array}{ll}
			{\bv^*}''(z) + \mathcal{S}::\bv^*(z) + \mathcal{F}(\bv^*(z)) = 0 &z \in (-R, R),\\
			\bv^*(z) = \mathcal{P}^*(z) &z = \pm R,\\
			{\bv^*}'(z) = \mathcal{Q}^*(z) &z = \pm R.
		\end{array}
	\right.
	\label{nonoise}
\end{equation}
In this section, we assume the existence of the solution $\bv^*$ to \eqref{nonoise}.
We now consider the case when noise is introduced to the data.
Let $\delta > 0$ be the noise level. 
That means,
\begin{equation}
	\max_{z \in \{-R, R\}}\big\{|\mathcal{P}(z) - \mathcal{P}^*(z), |\mathcal{Q}(z) - \mathcal{Q}^*(z)|\big\} < \delta.
	\label{noisedelta}
\end{equation}
\begin{Remark}	[Noise model]
	In this section, for simplicity, we assume that noise is introduced into the indirect data 
$\mathcal{P}(\pm R)$ and $\mathcal{Q}(\pm R)$ as in \eqref{noisedelta}  rather than to the direct data, 
$f$ and $g$. This assumption serves theoretical purposes only. In our computational study, we study the more realistic case where the direct data $f^*$ and $g^*$ are impacted by noise as in \eqref{5.4} and \eqref{5.5}.
 Recall that the the entries $P_{\bmm}(\pm R)$ and $Q_{\bmm}(\pm R)$ of indirect data
$\mathcal{P}(\pm R)$ and $\mathcal{Q}(\pm R)$ are computed by the knowledge of the derivatives of $f$ and $g$ via \eqref{2.1414} and \eqref{2.15}. 
Given that differentiating noisy data presents significant challenges and can greatly amplify errors, even minor noise in 
$f$
and 
$g$
can lead to substantial inaccuracies in 
$P_{\bmm}(\pm R)$ and $Q_{\bmm}(\pm R)$. To address this issue, we employ a novel differentiation approach as presented in \cite{NguyenLeNguyenKlibanov:2023}. In \cite{NguyenLeNguyenKlibanov:2023}, this method has been demonstrated to have superior stability compared to traditional methods like finite difference, cubic splines, or the Tikhonov regularization technique.
\end{Remark}
Consider the space of admissible solutions
\[
	H = \Big\{
		\bm{\varphi} = [\varphi_{\bmm}]_{\bmm \leqq \bN}: \varphi_{\bmm} \in H^2(-R, R)  \mbox{ for all } \bmm \leqq \bN
	\Big\}  = H^2(-R, R)^{|\bN|}
\]
where $|\bN|$ is as in \eqref{bN}.
Fix an arbitrary large number $M > 0$, define the close ball in $H$
\begin{equation}
	B_M = \{\bm{\varphi}: \|\bm{\varphi}\|_{H^2(-R, R)^{|\bN|}} \leq M\}.
\end{equation}
For each ${\bf v}  \in B_M$, define the functional
\begin{multline}
	J_{\bf v}^{(\lambda, \epsilon)}(\bm{\varphi})
	= 
	\int_{-R}^R e^{2\lambda (z - z_0)^{-2}}\Big|\bm{\varphi}'' + \mathcal{S}::\bm{\varphi}
	+\mathcal{F}(\bv)\Big|^2dz
	+ \lambda^4 e^{2\lambda (R - z_0)^{-2}} \left|\bm{\varphi}(R) - \mathcal{P}(R)\right|^2
	\\
	+
	 \lambda^4 e^{2\lambda (-R - z_0)^{-2}} \left|\bm{\varphi}(-R) - \mathcal{P}(-R)\right|^2
	 + \lambda^4 e^{2\lambda (R - z_0)^{-2}} \left|\bm{\varphi}'(R) - \mathcal{Q}(R)\right|^2
	 \\
	+
	 \lambda^4 e^{2\lambda (-R - z_0)^{-2}} \left|\bm{\varphi}'(-R) - \mathcal{Q}(-R)\right|^2
	+ \epsilon\| \bm{\varphi}\|_{H^2(-R, R)^{|\bN|}}^2
	\label{3.13}
\end{multline}
for all $\bm{\varphi}$
where $\lambda \geq \lambda_0$ and $\epsilon > 0$ will be chosen later.
For all ${\bf v} \in B_M$, the functional $J_{\bf v}^{(\lambda, \epsilon)}$ is uniformly convex in the close and convex set $B_M$ of $H$.
It has a unique minimizer. We define the map
$\Phi^{(\lambda, \epsilon)}: B_M \to B_M$ that sends ${\bf v}$ to such a minimizer. More precisely,
\begin{equation}
	\Phi^{(\lambda, \epsilon)}({\bf v})
	=
\underset{\bm{\varphi} \in B_M}{{\rm argmin}} J_{\bf v}^{(\lambda, \epsilon)}(\bm{\varphi})
\quad \mbox{for all }
{\bf v} \in B_M.
\label{optimization}
\end{equation}
We define the sequence $\{\bv^{(k)}\}_{k \geq 0}$ as follows:
\begin{equation}
	\left\{
		\begin{array}{ll}
			\bv^{(0)} \mbox{ chosen arbitrarily in } B_M 
			\\
			\bv^{(k+1)} = \Phi^{(\lambda, \epsilon)}(\bv^{(k)}) \quad k \geq 0.
		\end{array}
	\right.
	\label{sequence_v}
\end{equation}
The following theorem guarantees the convergence of the sequence $\{\bv^{(k)}\}_{k \geq 0}$ to $\bv^*$.
\begin{Theorem}	
	Let $M$ be a large number such that both $\bv^*$ and $\bv^{(0)}$ are in $B_M$.
	Let $\lambda_0$ be the number in Lemma \ref{lem31}. 
	Then, there exists $\lambda_1 > \lambda_0$ depending only on $M,$ $p$, $N,$ $R$, and $[{\bf P}_{\bmm \leqq \bN}]$ such that
	\begin{multline}
	\int_{-R}^R e^{2\lambda (z - z_0)^{-2}}|\bv^{(k+1)} - \bv^*|^2dz
	\leq
	\Big(\frac{C}{\lambda^3}\Big)^{k+1}
	\int_{-R}^R  e^{2\lambda(z - z_0)^{-2}}
	\big|\bv^{(0)} - \bv^*\big|^2dz
	\\
	+ \frac{C/\lambda^3}{1 - C/\lambda^3}
	\Big[
		\lambda^4e^{2\lambda(R - z_0)^{-2}}|\mathcal{P}(R) -\mathcal{P}^*(R))|^2
	+ \lambda^4e^{2\lambda(-R - z_0)^{-2}}|\mathcal{P}(-R) - \mathcal{P}^*(-R)|^2
	\\
	+\lambda^4e^{2\lambda(R - z_0)^{-2}}|\mathcal{Q}(R) -\mathcal{Q}^*(R)|^2
	+ \lambda^4e^{2\lambda(-R - z_0)^{-2}}|\mathcal{Q}(-R) -  \mathcal{Q}^*(-R)|^2
	+\epsilon\| \bv^{*}\|_{H^2(-R, R)^{|\bN|}}^2
	\Big]
	\label{4.19}
\end{multline}
for all $k \geq  0.$
\label{thmPicard}	
\end{Theorem}

\begin{Remark}
	Theorem \ref{thmPicard} and its proof are stated and proved using similar arguments in \cite{LeCON2023, LeNguyen:jiip2022, Nguyen:AVM2023}. However, we still need some important modifications:
	\begin{enumerate}
		\item The nonlinearity 
$\mathcal{F}$ in \cite{LeCON2023, LeNguyen:jiip2022, Nguyen:AVM2023} needs to satisfy the Lipschitz condition. However, the function 
$\mathcal{F}$ in the current work does not meet this requirement. To address this issue, it is necessary to confine the computational domain to a bounded set 
$B_M$ for an arbitrarily large number $M$.
Within this bounded domain, the Lipschitz condition is automatically satisfied.
		\item In \cite{LeCON2023, LeNguyen:jiip2022}, the analysis of noise was not explored, whereas it was somewhat examined in \cite{Nguyen:AVM2023}. By ``somewhat," it means that in \cite{Nguyen:AVM2023}, a technical condition had to be imposed. The noise in the Dirichlet observations and the noise in the Neumann measurements are not independent. Specifically, it was assumed that the noise in the Dirichlet observation is the trace of a function, and the noise in the Neumann measurement needs to be the trace of that function's normal derivative. Given that this circumstance is somewhat impractical, we opt to relax it in the present paper.
	\end{enumerate}
\end{Remark}

\begin{proof}[Proof of Theorem \ref{thmPicard}]	
In the proof, we will employ the dot product 
\[
	\bm{\varphi} \cdot {\bf h} = \sum_{\bmm \leqq \bN} \varphi_{\bmm} h_{\bmm}
\]
for all $\bm{\varphi} = [\varphi_{\bmm}]_{\bmm \leq \bN}$ and ${\bf h}= [h_{\bmm}]_{\bmm \leq \bN}$ in $H$.
Fix $k \geq 0$. 
Set
\begin{equation*}
	{\bf h} = \bv^{(k+1)} - \bv^*.
\end{equation*}
Since $\bv^{(k + 1)} $ is the minimizer of $J^{\lambda, \epsilon}_{\bv^{(k)}}$ in $B_M$ and $\bv^*$ is in the interior of $B_M$, we have
\begin{multline}
	\int_{-R}^R e^{2\lambda(z - z_0)^{-2}}
	\big( {\bv^{(k+1)}}''(z)	
	+ \mathcal{S}::\bv^{(k + 1)}(z)
	+ \mathcal{F}(\bv^{(k)})\big)\cdot
	\big(
		{\bf h}''(z)+ \mathcal{S}::{\bf h}(z)
	\big)dz
	\\
	+
	\lambda^4e^{2\lambda(R - z_0)^{-2}}(\bv^{(k+1)}(R) - \mathcal{P}(R))\cdot{\bf h}(R)
	+
	\lambda^4e^{2\lambda(-R - z_0)^{-2}}(\bv^{(k+1)}(-R) - \mathcal{P}(-R))\cdot{\bf h}(-R)
	\\
	\lambda^4e^{2\lambda(R - z_0)^{-2}}({\bv^{(k+1)}}'(R) - \mathcal{Q}(R))\cdot{\bf h}'(R)
	+
	\lambda^4e^{2\lambda(-R - z_0)^{-2}}({\bv^{(k+1)}}'(-R) - \mathcal{Q}(-R))\cdot{\bf h}'(-R)
	\\
	+\epsilon\langle \bv^{(k+1)}, {\bf h}\rangle_{H^2(-R, R)^{|\bN|}} \leq 0.
	\label{4.6}
\end{multline}
On the other hand, since $\bv^*$ is the true solution to \eqref{nonoise}, we have
\begin{multline}
	\int_{-R}^R e^{2\lambda(z - z_0)^{-2}}
	\big( {\bv^{*}}''(z)	
	+ \mathcal{S}::\bv^{*}(z)
	+ \mathcal{F}(\bv^{*})\big)\cdot
	\big(
		{\bf h}''(z)+ \mathcal{S}::{\bf h}(z)
	\big)dz
	\\
	+
	\lambda^4e^{2\lambda(R - z_0)^{-2}}(\bv^{*}(R) - \mathcal{P}^*(R))\cdot{\bf h}(R)
	+
	\lambda^4e^{2\lambda(-R - z_0)^{-2}}(\bv^{*}(-R) - \mathcal{P}^*(-R))\cdot{\bf h}(-R)
	\\
	\lambda^4e^{2\lambda(R - z_0)^{-2}}({\bv^{*}}'(R) - \mathcal{Q}^*(R))\cdot{\bf h}'(R)
	+
	\lambda^4e^{2\lambda(-R - z_0)^{-2}}({\bv^{*}}'(-R) - \mathcal{Q}^*(-R))\cdot{\bf h}'(-R)
	\\
	+\epsilon\langle \bv^{*}, {\bf h}\rangle_{H^2(-R, R)^{|\bN|}} = \epsilon\langle \bv^{*}, {\bf h}\rangle_{H^2(-R, R)^{|\bN|}}.
	\label{4.7}
\end{multline}
Subtracting \eqref{4.6} from \eqref{4.7} gives
\begin{multline}
	\int_{-R}^R e^{2\lambda(z - z_0)^{-2}}
	\big( {\bv^{(k+1)}}''(z)	 - {\bv^{*}}''(z)	
	+ \mathcal{S}::(\bv^{(k + 1)}(z) - {\bv^{*}}''(z)	)
	+ \mathcal{F}(\bv^{(k)}) - \mathcal{F}(\bv^{*})\big)\cdot
	\big(
		{\bf h}''(z)+ \mathcal{S}::{\bf h}(z)
	\big)dz
	\\
	+
	\lambda^4e^{2\lambda(R - z_0)^{-2}}(\bv^{(k+1)}(R) - \bv^{*}(R) )\cdot{\bf h}(R)
	+
	\lambda^4e^{2\lambda(-R - z_0)^{-2}}(\bv^{(k+1)}(-R) - \bv^{*}(-R))\cdot{\bf h}(-R)
	\\
	+
	\lambda^4e^{2\lambda(R - z_0)^{-2}}({\bv^{(k+1)}}'(R) - {\bv^{*}}'(R) )\cdot{\bf h}'(R)
	+
	\lambda^4e^{2\lambda(-R - z_0)^{-2}}({\bv^{(k+1)}}'(-R) - {\bv^{*}}'(-R))\cdot{\bf h}'(-R)
		\\
	+\epsilon\langle \bv^{(k+1)} - \bv^*, {\bf h}\rangle_{H^2(-R, R)^{|\bN|}} 
	\\
	\leq 
	\lambda^4e^{2\lambda(R - z_0)^{-2}}(\mathcal{P}(R) -\mathcal{P}^*(R))\cdot{\bf h}(R)
	+ \lambda^4e^{2\lambda(-R - z_0)^{-2}}(\mathcal{P}(-R) - \mathcal{P}^*(-R))\cdot{\bf h}(-R)
	\\
	+\lambda^4e^{2\lambda(R - z_0)^{-2}}(\mathcal{Q}(R) -\mathcal{Q}^*(R))\cdot{\bf h}'(R)
	+ \lambda^4e^{2\lambda(-R - z_0)^{-2}}(\mathcal{Q}(-R) -  \mathcal{Q}^*(-R))\cdot{\bf h}'(-R)
	\\
	-\epsilon\langle \bv^{*}, {\bf h}\rangle_{H^2(-R, R)^{|\bN|}}.
	\label{4.8}
\end{multline}
Recalling ${\bf h} = \bv^{(k + 1)} - \bv^*$, we have
\begin{multline}
	\int_{-R}^R e^{2\lambda(z - z_0)^{-2}}
	\big| {\bf h}''(z)	
	+ \mathcal{S}::{\bf h}(z)
	\big|^2dz
	+
	\lambda^4e^{2\lambda(R - z_0)^{-2}}\ 
	|{\bf h}(R)|^2
	+
	\lambda^4e^{2\lambda(-R - z_0)^{-2}}|{\bf h}(-R)|^2	
	\\
	+
	\lambda^4e^{2\lambda(R - z_0)^{-2}}|{\bf h}'(R)|^2	+
	\lambda^4e^{2\lambda(-R - z_0)^{-2}}|{\bf h}'(-R)|^2
	+\epsilon \|{\bf h}\|_{H^2(-R, R)^{|\bN|}}^2 
	\\
	\leq 
	-\int_{-R}^R  e^{2\lambda(z - z_0)^{-2}}\big(
	\mathcal{F}(\bv^{(k)}) - \mathcal{F}(\bv^{*})\big)\cdot
	\big(
		{\bf h}''(z)+ \mathcal{S}::{\bf h}(z)
	\big)dz
	+
	\lambda^4e^{2\lambda(R - z_0)^{-2}}(\mathcal{P}(R) -\mathcal{P}^*(R))\cdot{\bf h}(R)
	\\
	+ \lambda^4e^{2\lambda(-R - z_0)^{-2}}(\mathcal{P}(-R) - \mathcal{P}^*(-R))\cdot{\bf h}(-R)
	+\lambda^4e^{2\lambda(R - z_0)^{-2}}(\mathcal{Q}(R) -\mathcal{Q}^*(R))\cdot{\bf h}'(R)
	\\
	+ \lambda^4e^{2\lambda(-R - z_0)^{-2}}(\mathcal{Q}(-R) -  \mathcal{Q}^*(-R))\cdot{\bf h}'(-R)
	-\epsilon\langle \bv^{*}, {\bf h}\rangle_{H^2(-R, R)^{|\bN|}}.
	\label{4.10}
\end{multline}
Rearranging terms \eqref{4.10} and using the inequality $|ab| \leq \frac{1}{2}(a^2 + b^2)$, we have
\begin{multline}
	\int_{-R}^R e^{2\lambda(z - z_0)^{-2}}
	\big| {\bf h}''(z)	
	+ \mathcal{S}::{\bf h}(z)
	\big|^2dz
	+
	\lambda^4e^{2\lambda(R - z_0)^{-2}}\ 
	|{\bf h}(R)|^2
	+
	\lambda^4e^{2\lambda(-R - z_0)^{-2}}|{\bf h}(-R)|^2	
	\\
	+
	\lambda^4e^{2\lambda(R - z_0)^{-2}}|{\bf h}'(R)|^2	+
	\lambda^4e^{2\lambda(-R - z_0)^{-2}}|{\bf h}'(-R)|^2
	+\epsilon \|{\bf h}\|_{H^2(-R, R)^{|\bN|}}^2 
	\\
	\leq 
	-\int_{-R}^R  e^{2\lambda(z - z_0)^{-2}}
	\big|\mathcal{F}(\bv^{(k)}) - \mathcal{F}(\bv^{*})\big|^2dz
	+
	\lambda^4e^{2\lambda(R - z_0)^{-2}}|\mathcal{P}(R) -\mathcal{P}^*(R))|^2
	\\
	+ \lambda^4e^{2\lambda(-R - z_0)^{-2}}|\mathcal{P}(-R) - \mathcal{P}^*(-R)|^2
	+\lambda^4e^{2\lambda(R - z_0)^{-2}}|\mathcal{Q}(R) -\mathcal{Q}^*(R)|^2
	\\
	+ \lambda^4e^{2\lambda(-R - z_0)^{-2}}|\mathcal{Q}(-R) -  \mathcal{Q}^*(-R)|^2
	+\epsilon\| \bv^{*}\|_{H^2(-R, R)^{|\bN|}}^2.
	\label{4.11}
\end{multline}
Apply the inequality $(a + b)^2 \geq \frac{1}{2}a^2 - b^2$ for the first term of \eqref{4.11}.
We have
\begin{multline}
	\frac{1}{2}\int_{-R}^R e^{2\lambda(z - z_0)^{-2}}
	\big| {\bf h}''(z)	
	\big|^2dz
	+
	\lambda^4e^{2\lambda(R - z_0)^{-2}}\ 
	|{\bf h}(R)|^2
	+
	\lambda^4e^{2\lambda(-R - z_0)^{-2}}|{\bf h}(-R)|^2	
	\\
	+
	\lambda^4e^{2\lambda(R - z_0)^{-2}}|{\bf h}'(R)|^2	+
	\lambda^4e^{2\lambda(-R - z_0)^{-2}}|{\bf h}'(-R)|^2
	+\epsilon \|{\bf h}\|_{H}^2 
	\leq 
	2\int_{-R}^R  e^{2\lambda(z - z_0)^{-2}} |\mathcal{S}::{\bf h}|^2dz
	\\
	-\int_{-R}^R  e^{2\lambda(z - z_0)^{-2}}
	\big|\mathcal{F}(\bv^{(k)}) - \mathcal{F}(\bv^{*})\big|^2dz
	+
	\lambda^4e^{2\lambda(R - z_0)^{-2}}|\mathcal{P}(R) -\mathcal{P}^*(R))|^2
	\\
	+ \lambda^4e^{2\lambda(-R - z_0)^{-2}}|\mathcal{P}(-R) - \mathcal{P}^*(-R)|^2
	+\lambda^4e^{2\lambda(R - z_0)^{-2}}|\mathcal{Q}(R) -\mathcal{Q}^*(R)|^2
	\\
	+ \lambda^4e^{2\lambda(-R - z_0)^{-2}}|\mathcal{Q}(-R) -  \mathcal{Q}^*(-R)|^2
	+\epsilon\| \bv^{*}\|_{H^2(-R, R)^{|\bN|}}^2.
	\label{4.12}
\end{multline}

Since $B_M$ is bounded in $H$, by the Sobolev embedding theorem in 1D, $B_M$ is bounded in $C([-R, R])^{|\bN|}$.
We can find a number $C$ depending only on $M$, $\mathcal{F}$ (and hence $p,$ $N$, $R$, $[{\bf P}_{\bmm}]_{\bmm \leqq \bN}$) such that
\begin{equation}
|\mathcal{F}(\bv^{(k)}(z)) - \mathcal{F}(\mathcal{\bv^*})(z)|
	\leq 
	C |\bv^{(k)}(z) -\mathcal{\bv^*}(z)|
	\quad
	\mbox{for all } z \in (-R, R).
	\label{4.13}
\end{equation}
Combining \eqref{4.12} and \eqref{4.13}, we can find a constant $C$ depending only $M$, $p$, $N$, $R$, $[{\bf P}_{\bmm}]_{\bmm \leqq \bN}$ such that
\begin{multline}
	\int_{-R}^R e^{2\lambda(z - z_0)^{-2}}
	\big| {\bf h}''(z) \big|^2dz
	+ \lambda^4e^{2\lambda(R - z_0)^{-2}}\ 
	|{\bf h}(R)|^2
	+
	\lambda^4e^{2\lambda(-R - z_0)^{-2}}|{\bf h}(-R)|^2	
	\\
	+
	\lambda^4e^{2\lambda(R - z_0)^{-2}}|{\bf h}'(R)|^2	+
	\lambda^4e^{2\lambda(-R - z_0)^{-2}}|{\bf h}'(-R)|^2
	+\epsilon \|{\bf h}\|_{H^2(-R, R)^{|\bN|}}^2 
	\\
	\leq C \Big[
		\int_{-R}^R  e^{2\lambda(z - z_0)^{-2}}
	\big|{\bf h}\big|^2dz
	+
	\int_{-R}^R  e^{2\lambda(z - z_0)^{-2}}
	\big|\bv^{(k)} - \bv^*\big|^2dz
	+
	\lambda^4e^{2\lambda(R - z_0)^{-2}}|\mathcal{P}(R) -\mathcal{P}^*(R))|^2
	\\
	+ \lambda^4e^{2\lambda(-R - z_0)^{-2}}|\mathcal{P}(-R) - \mathcal{P}^*(-R)|^2
	+\lambda^4e^{2\lambda(R - z_0)^{-2}}|\mathcal{Q}(R) -\mathcal{Q}^*(R)|^2
	+ \lambda^4e^{2\lambda(-R - z_0)^{-2}}|\mathcal{Q}(-R) -  \mathcal{Q}^*(-R)|^2
	\\
	+\epsilon\| \bv^{*}\|_{H^2(-R, R)^{|\bN|}}^2
	\Big].
	\label{4.14}
\end{multline}
Applying the Carleman estimate in Lemma \ref{lem31} for each entry of ${\bf h}$, we have
\begin{multline}
		\int_{-R}^Re^{2\lambda (z - z_0)^{-2}} |{\bf h}''(z)|^2 dz
		\geq  
		-Ce^{2\lambda (R-z_0)^{-2}}(\lambda^3 |{\bf h}(R)|^2 + \lambda |{\bf h}'(R)|^2) 
		\\
		- C e^{2\lambda (-R-z_0)^{-2}} (\lambda^3 |{\bf h}(-R)|^2 + \lambda |{\bf h}'(-R)|^2)
		+C \lambda^3 \int_{-R}^R e^{2\lambda (z - z_0)^{-2}}|{\bf h}(z)|^2dz		
		\\
		+C \lambda \int_{-R}^R e^{2\lambda (z - z_0)^{-2}} |{\bf h}'(z)|^2dz.
		\label{4.15}
	\end{multline}
Combining \eqref{4.14} and \eqref{4.15} and noting that $\lambda^4 \gg \lambda^3$	, we have
\begin{multline}
	\lambda^3 \int_{-R}^R e^{2\lambda (z - z_0)^{-2}}|{\bf h}(z)|^2dz
	+ 
	\lambda \int_{-R}^R e^{2\lambda (z - z_0)^{-2}} |{\bf h}'(z)|^2dz
	+ \lambda^4e^{2\lambda(R - z_0)^{-2}}\ 
	|{\bf h}(R)|^2
	\\
	+
	\lambda^4e^{2\lambda(-R - z_0)^{-2}}|{\bf h}(-R)|^2	
	+
	\lambda^4e^{2\lambda(R - z_0)^{-2}}|{\bf h}'(R)|^2	+
	\lambda^4e^{2\lambda(-R - z_0)^{-2}}|{\bf h}'(-R)|^2
	\\
	\leq
	C \Big[
		\int_{-R}^R  e^{2\lambda(z - z_0)^{-2}}
	\big|{\bf h}\big|^2dz
	+
	\int_{-R}^R  e^{2\lambda(z - z_0)^{-2}}
	\big|\bv^{(k)} - \bv^*\big|^2dz
	+
	\lambda^4e^{2\lambda(R - z_0)^{-2}}|\mathcal{P}(R) -\mathcal{P}^*(R))|^2
	\\
	+ \lambda^4e^{2\lambda(-R - z_0)^{-2}}|\mathcal{P}(-R) - \mathcal{P}^*(-R)|^2
	+\lambda^4e^{2\lambda(R - z_0)^{-2}}|\mathcal{Q}(R) -\mathcal{Q}^*(R)|^2
	\\
	+ \lambda^4e^{2\lambda(-R - z_0)^{-2}}|\mathcal{Q}(-R) -  \mathcal{Q}^*(-R)|^2
	+\epsilon\| \bv^{*}\|_{H^2(-R, R)^{|\bN|}}^2
	\Big].
	\label{4.16}
\end{multline}
It follows from \eqref{4.16} and the fact ${\bf h} = \bv^{(k+1)} - \bv^*$ that
\begin{multline}
	\int_{-R}^R e^{2\lambda (z - z_0)^{-2}}|\bv^{(k+1)} - \bv^*|^2dz
	\leq
	\frac{C}{\lambda^3}
	\Big[
		\int_{-R}^R  e^{2\lambda(z - z_0)^{-2}}
	\big|\bv^{(k)} - \bv^*\big|^2dz
	+ 
	\lambda^4e^{2\lambda(R - z_0)^{-2}}|\mathcal{P}(R) -\mathcal{P}^*(R))|^2
	\\
	+ \lambda^4e^{2\lambda(-R - z_0)^{-2}}|\mathcal{P}(-R) - \mathcal{P}^*(-R)|^2
	+\lambda^4e^{2\lambda(R - z_0)^{-2}}|\mathcal{Q}(R) -\mathcal{Q}^*(R)|^2
	\\
	+ \lambda^4e^{2\lambda(-R - z_0)^{-2}}|\mathcal{Q}(-R) -  \mathcal{Q}^*(-R)|^2
	+\epsilon\| \bv^{*}\|_{H^2(-R, R)^{|\bN|}}^2
	\Big].
	\label{4.17}
\end{multline}
Recall that ${\bf h} = \bv^{(k+1)} - \bv^*$.
Applying \eqref{4.17} when $k + 1$ is replaced by $k$ and combining the resulting estimate with \eqref{4.17}, we have
\begin{multline*}
	\int_{-R}^R e^{2\lambda (z - z_0)^{-2}}|\bv^{(k+1)} - \bv^*|^2dz
	\leq
	\frac{C}{\lambda^3}
	\Big[
		\frac{C}{\lambda^3}
	\Big[
		\int_{-R}^R  e^{2\lambda(z - z_0)^{-2}}
	\big|\bv^{(k-1)} - \bv^*\big|^2dz
	\\
	+ 
	\lambda^4e^{2\lambda(R - z_0)^{-2}}|\mathcal{P}(R) -\mathcal{P}^*(R))|^2
	+ \lambda^4e^{2\lambda(-R - z_0)^{-2}}|\mathcal{P}(-R) - \mathcal{P}^*(-R)|^2
	\\
	+\lambda^4e^{2\lambda(R - z_0)^{-2}}|\mathcal{Q}(R) -\mathcal{Q}^*(R)|^2
	+ \lambda^4e^{2\lambda(-R - z_0)^{-2}}|\mathcal{Q}(-R) -  \mathcal{Q}^*(-R)|^2
	\\
	+\epsilon\| \bv^{*}\|_{H^2(-R, R)^{|\bN|}}^2
	\Big]
	+ 
	\lambda^4e^{2\lambda(R - z_0)^{-2}}|\mathcal{P}(R) -\mathcal{P}^*(R))|^2
	+ \lambda^4e^{2\lambda(-R - z_0)^{-2}}|\mathcal{P}(-R) - \mathcal{P}^*(-R)|^2
	\\
	+\lambda^4e^{2\lambda(R - z_0)^{-2}}|\mathcal{Q}(R) -\mathcal{Q}^*(R)|^2
	+ \lambda^4e^{2\lambda(-R - z_0)^{-2}}|\mathcal{Q}(-R) -  \mathcal{Q}^*(-R)|^2
	+\epsilon\| \bv^{*}\|_{H^2(-R, R)^{|\bN|}}^2
	\Big].
\end{multline*}
Continuing the procedure, we obtain
\begin{multline}
	\int_{-R}^R e^{2\lambda (z - z_0)^{-2}}|\bv^{(k+1)} - \bv^*|^2dz
	\leq
	\Big(\frac{C}{\lambda^3}\Big)^{k+1}
	\int_{-R}^R  e^{2\lambda(z - z_0)^{-2}}
	\big|\bv^{(0)} - \bv^*\big|^2dz
	\\
	+ \sum_{i = 1}^k \Big(\frac{C}{\lambda^3}\Big)^{k}
	\Big[
		\lambda^4e^{2\lambda(R - z_0)^{-2}}|\mathcal{P}(R) -\mathcal{P}^*(R))|^2
	+ \lambda^4e^{2\lambda(-R - z_0)^{-2}}|\mathcal{P}(-R) - \mathcal{P}^*(-R)|^2
	\\
	+\lambda^4e^{2\lambda(R - z_0)^{-2}}|\mathcal{Q}(R) -\mathcal{Q}^*(R)|^2
	+ \lambda^4e^{2\lambda(-R - z_0)^{-2}}|\mathcal{Q}(-R) -  \mathcal{Q}^*(-R)|^2
	+\epsilon\| \bv^{*}\|_{H^2(-R, R)^{|\bN|}}^2
	\Big].
	\label{4.18}
\end{multline}
Choose $\lambda >  \lambda_1 > \lambda_0$ such that $C/\lambda^3 \in (0, 1)$ for some $\lambda_1$ depending only on $C$ and therefore only on $M,$ $p$, $N,$ $R$, and $[{\bf P}_{\bmm \leqq \bN}]$.
Estimate \eqref{4.19} is a direct consequence of \eqref{4.18}.
\end{proof}

\begin{Remark}
	Estimate \eqref{4.19} is interesting in the sense that when the data has noise, although the over-determined problem \eqref{2.12} might have no solution, we are still able to provide a resonably accurate numerical solution.
	In fact, the sequence $\{\bv^{(k)}\}_{k \geq 0}$ is well-defined regardless whether or not \eqref{2.12} is solvable. 
	This sequence is defined via minimizing a strictly convex cost functional as in \eqref{optimization} when $\bv$ is replaced by $\bv^{(k)}$, $k \geq 0$.
	Due to \eqref{noisedelta} and the result of Theorem \ref{thmPicard} in  \eqref{4.7}, when $\lambda > \lambda_1 $ is fixed so that $C/\lambda^3 \in (0, 1)$, we have
	\begin{equation}
	\|\bv^{(k+1)} - \bv^*\|^2_{L^2(-R, R)}
	\leq C_1\Big[
	\Big(\frac{C}{\lambda^3}\Big)^{k+1}
		\|\bv^{(0)} - \bv^*\|^2_{L^2(-R, R)}
	+ \frac{\delta^2 C/\lambda^3}{1 - C/\lambda^3}
	+\epsilon\| \bv^{*}\|_{H^2(-R, R)^{|\bN|}}^2
	\Big]
	\label{4.20}
\end{equation}
	where $C$ is the constant in Theorem \ref{thmPicard} and $C_1$ is a constant depending only on $\lambda,$ $z_0$, and $R$.
	On the other hand, \eqref{4.20} guarantees the global convergence. 
	Although $\bv^{(0)}$ is not a good initial guess of $\bv^*$, the approximating sequence converges to a numerical solution with the error $O(\delta + \sqrt{\epsilon} \|\bv^*\|_{H^2(-R, R)^{|\bN|}}).$	
	\label{rem4.2}
\end{Remark}	

Motivated from \eqref{2.5}, define
\[
	v^{(k)}(\x', z, t) = \sum_{n_1 = 1}^{N_1}
	\dots
	\sum_{n_{d - 1}}^{N_{d - 1}}
	\sum_{n_t = 1}^{N_t}
	v_{(n_1, \dots, n_{d - 1}, n_t)} (z)
	\Psi_{n_1}(x_1) \dots \Psi_{n_{d - 1}}(x_{d - 1})\psi_{n_t}(t)
\] for all $(\x, t) = (\x', z, t) \in \Omega \times (0, T)$.
Due to \eqref{2.3}, set
\begin{equation}
	c^{(k)}(\x) = \frac{v^{(k)}(\x, 0) - \Delta p(\x)}{p(\x)}
\end{equation}
for all $\x \in \Omega.$
By Remark \ref{rem4.2} and Theorem \ref{thmPicard}, we can find $\lambda_1$ as in Theorem \ref{thmPicard} such that for all $\lambda > \lambda_1$, we have
	\begin{equation}
	\|c^{(k+1)} - c^*\|^2_{L^2(-R, R)}
	\leq C_1\Big[
	\Big(\frac{C}{\lambda^3}\Big)^{k+1}
		\|\bv^{(0)} - \bv^*\|^2_{L^2(-R, R)}
	+ \frac{\delta^2 C/\lambda^3}{1 - C/\lambda^3}
	+\epsilon\| \bv^{*}\|_{H^2(-R, R)^{|\bN|}}^2 
	\Big]
	\label{5.7.1}
\end{equation}
where $C_1$ is a constant depending only on $\lambda,$ $z_0$, $R$, $\bN$, and $[\bP_{\bn}]_{\bn \leqq \bN}.$
Here $c^*$ is the true coefficient defined by
\[
		c^{*}(\x) = \frac{v^{*}(\x, 0) - \Delta p(\x)}{p(\x)}
\]
for all $\x \in \Omega$.
 
The analysis above leads to Algorithm \ref{alg} to solve the inverse problem under consideration.
Having the data $f$ and $g$ in hand, we can follow Algorithm \ref{alg} to compute a numerical solution to the inverse problem under consideration.
\begin{algorithm}[!ht]
\caption{\label{alg}The procedure to compute the numerical solution to \eqref{2.12}}
	\begin{algorithmic}[1]
	\State \label{s_chooseN} Choose cut-off numbers $N_1, \dots, N_{d-1}$, and $N_t$. 
	Set $\bN = (N_1, N_2, \dots, N_{d-1}, N_t)$
	\State \label{s1} Choose Carleman parameters $z_0,$  and $\lambda$, and a regularization parameter $\epsilon$.
	Choose a large number $M$.
	 \State \label{s2} Set $n = 0$. Choose an arbitrary initial solution $\bv^{(0)} \in B_M.$
%	 \For{$n = 1$ to $\overline n$ (for some positive integer $\overline n$)}
		\State \label{step update} 		
		Compute $\bv_{(k + 1)} = \Phi^{(\lambda, \epsilon)}(\bv^{(k)})$ where $ \Phi^{(\lambda, \epsilon)}$ is defined in \eqref{optimization}.
%	 \EndFor
	\If {$\|\bv_{(k + 1)} - \bv_{(k)}\|_{L^2(-R, R)} > \kappa_0$ (for some fixed number $\kappa_0 > 0$)} \label{if}
		\State Replace $k$ by $k + 1.$
		\State\label{step7} Go back to Step \ref{step update}.
	\Else	
		\State Set the computed solution $\bv^{\rm comp} = \bv_{(k + 1)}.$
	\EndIf
	\State \label{s11} Write $\bv^{\rm comp} = [v_{\bmm}^{(k+1)}]_{\bmm \leqq \bN}$ and set the desired solution as in \eqref{2.12}.
	\State Due to \eqref{2.5}, compute 
	\[
		v^{\rm comp} (\x, t) = v(\x', z, t) = \sum_{n_1 = 1}^{N_1}
	\dots
	\sum_{n_{d - 1}}^{N_{d - 1}}
	\sum_{n_t = 1}^{N_t}
	v_{(n_1, \dots, n_{d - 1}, n_t)}^{\rm comp}  (z)
	\Psi_{n_1}(x_1) \dots \Psi_{n_{d - 1}}(x_{d - 1})\psi_{n_t}(t)
	\]
	for $\x \in \Omega$, $t \in (0, T).$
	\State \label{s13} By \eqref{2.3}, we obtain the reconstructed coefficient $c$ as
	\begin{equation}
		c^{\rm comp}(\x) = \frac{v^{\rm comp}(\x, 0) - \Delta p(\x)}{p(\x)}
	\end{equation} for all $\x \in \Omega$.
\end{algorithmic}
\end{algorithm}

\section{Numerical study}\label{sec5}

To illustrate our method, we numerically study the inverse problem, Problem \ref{p}, in 2D. That means, for simplicity, we set $d = 2$. 
Let $\Omega = (-R, R)^2$ and therefore $\Omega = (-R, R)$. 
Rather than using the notation $\x'$, we write $x$ for an element of $\Omega'.$
In this section, we set $R = 1$.

\subsection{The forward problem}
To generate simulated data, we numerically compute the solution to \eqref{main_eqn}. Since we only need the data on $\Gamma,$ a part of $\partial \Omega$ and $t \in (0, T)$, it is not necessary to solve \eqref{main_eqn} on the whole unbounded domain $\R^d \times (0, \infty)$. 
Rather, we choose a domain $G = (-R_1, R_1)^2$ for some $R_1 > R$, in which $\Omega$ is compactly contained, and a positive number $T$.
We solve  
\begin{equation}
	\left\{
		\begin{array}{rcll}
			u_t(\x, t) &=& \Delta u(\x, t) + c(\x)u(\x, t) &(\x, t) \in G \times (0, T),\\
			u(\x, 0) &=& p(\x) &\x \in G
		\end{array}
	\right.
	\label{5.1}
\end{equation}
by the explicit method in the finite difference scheme.
More precisely, we discretize $G$ by arranging a grid of points
\[
	\mathcal{G} = \big\{
		\x_{i,j} = (x_i = -R_1 + (i - 1)d_\x , z_j = -R_1 + (j - 1)d_\x ): i, j = 1, \dots, N_\x^1 
	\big\} \subset \overline G.
\]
where $N_\x^1$ is an integer and $d_\x = 2R_1/(N_\x^1 - 1)$.
On the time domain, we arrange the partition of $[0, T]$ as
\[
	\mathcal{T} = \big\{
		t_l = (l - 1)d_t: l = 1, \dots, N_t
	\big\}
\]
where $N_t$ is an integer and $d_t = T/(N_t-1).$
We set $R = 1$, $R_1 = 3$, $T = 0.5$, $N_\x^1 = 241$, and $N_t = 4001$. 
The finite difference version of \eqref{5.1} is read as
\begin{multline}
	\frac{u(\x_{i,j}, t_{l+1}) - u(\x_{i,j}, t_l)}{d_t}
	 = \frac{u(\x_{i+1,j}, t_l) 
	+ u(\x_{i-1,j}, t_l)
	+u(\x_{i,j+1}, t_l)
	+u(\x_{i,j-1}, t_l)
	- 4u(\x_{i,j}, t_l)}{d_\x^2}	
	\\
	+ c(\x_{i, j}) u(\x_{i,j}, t_l)
	\label{5.2}
\end{multline}
for all $\x_{i, j} \in \mathcal{G}$, $i, j \in \{1, \dots, N_\x^1\}$, and $t_l \in \mathcal{T},$ $l \in \{1, \dots, N_t\}.$
It follows from \eqref{5.2} that
\begin{multline}
	u(\x_{i,j}, t_{l+1}) = u(\x_{i,j}, t_l)	+
	d_t\Big[\frac{u(\x_{i+1,j}, t_l) 
	+ u(\x_{i-1,j}, t_l)
	+u(\x_{i,j+1}, t_l)
	+u(\x_{i,j-1}, t_l)
	- 4u(\x_{i,j}, t_l)}{d_\x^2}	
	\\
	+ c(\x_{i,j}) u(\x_{i,j}, t_l)\Big].
	\label{5.3}
\end{multline}
So, given $u(\x_{i,j}, t_1) = p(\x_{i, j})$ for $i, j \in \{1, \dots, N_\x^1\}$, we can compute $u(\x_{i,j}, t_2)$ via \eqref{5.3}, and then continue get $u(\x_{i,j}, t_{N_t})$.
In our computation, we set $p(\x) = 2$ for all $\x \in G.$
We then easily extract the noiseless data $f^*$ and $g^*$ on $\Gamma$. We pretend not to know $f^*$ and $g^*$. We only know the noisy version of $f^*$ and $g^*$ as
\begin{equation}
	f = f^*(1 + \mbox{random numbers in } [-\delta, \delta]) 
	\label{5.4}
\end{equation}
and
\begin{equation}
	g = g^*(1 + \mbox{random numbers in } [-\delta, \delta]) 
	\label{5.5}
\end{equation}
where $\delta$ is the noise level.
In our computational program, we choose the initial function $u(\x, 0) = p(\x) = 2$, for all $\x \in \Omega$, $R_1= 3$ and $T = 0.5$.

\subsection{The implementation of Algorithm \ref{alg}}\label{sec52}

{\bf Step \ref{s_chooseN} of Algorithm \ref{alg}.} In 2D, we only need we determine  $N_1$ and $N_t$. These numbers  are chosen by a data-driven procedure.
Recall that the data is measured on 
$\Gamma$ defined as in \eqref{gamma}. The measurement set consists of two parts, namely $\Gamma = \Gamma^+ \cup \Gamma^-$ where
\[
	\Gamma^{+} = \{\x = (x, z): x \in (-R, R) \mbox{and } z = R\}
\] and
\[
	\Gamma^{-} = \{\x = (x, z): x \in (-R, R) \mbox{and } z = -R\}.
\]
One can examine the approximation formula
\begin{equation}
	u(x, z, t) 
	\simeq
	\sum_{n_1 = 1}^{N_1} \sum_{n_t = 1}^{N_t} u_{(n_1, n_t)}(z) \bP_{(n_1, n_t)}(x, t)
		\label{5.6}
\end{equation}
 for the data $u(\x, t) = u(x, z, t) = f(\x, t)$ on either $\Gamma^+$ or $\Gamma^-$.
 Note that \eqref{5.6} is an analog of \eqref{2.5} where $v$ is replaced by the function $u$.
 We do so on $\Gamma^-.$ Define the mismatch function
 \[
 	\varphi_{(N_1, N_t)}(x, t) =\Big|u(x, -R, t) - \sum_{n_1 = 1}^{N_1} \sum_{n_t = 1}^{N_t} u_{(n_1, n_t)}(-R) \bP_{(n_1, n_t)}(x, t)\Big|
 \] for $(x, t) \in \Gamma^- \times (0, T).$
 Then, we test the smallness of $\varphi(N_1, N_t)$ by manually and gradually increasing $N_1$ and $N_t$ until $\varphi(N_1, N_t)$ is sufficiently small.
 For example, we take the data in the numerical test 1 below. We increase those two numbers so that $\|\varphi_{(N_1, N_t)}\|_{L^{\infty}(\Gamma^- \times (0, T))} < 5\times10^{-4}$. By this, we find $N_1 = 15$ and $N_t = 10$. We chose these two numbers for all of our numerical tests. 
 In Figure \ref{ChooseN}, we display the graphs of the function $\varphi_{(N_1, N_t)}$ for some values of $N_1$ and $N_t.$

\begin{figure}[!ht]
\begin{center}
	\subfloat[$\varphi_{(5, 5)}$]{\includegraphics[width=.3\textwidth]{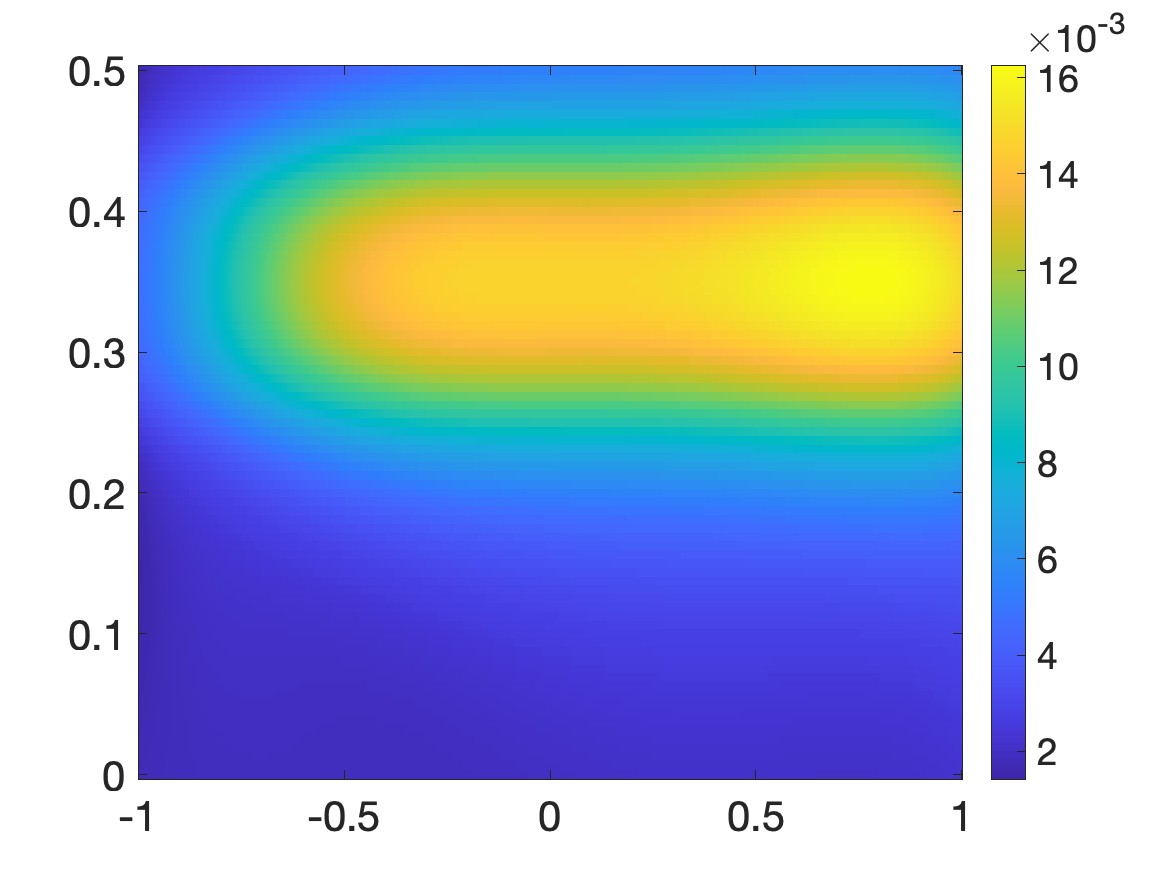}}
	\subfloat[$\varphi_{(15, 8)}$]{\includegraphics[width=.3\textwidth]{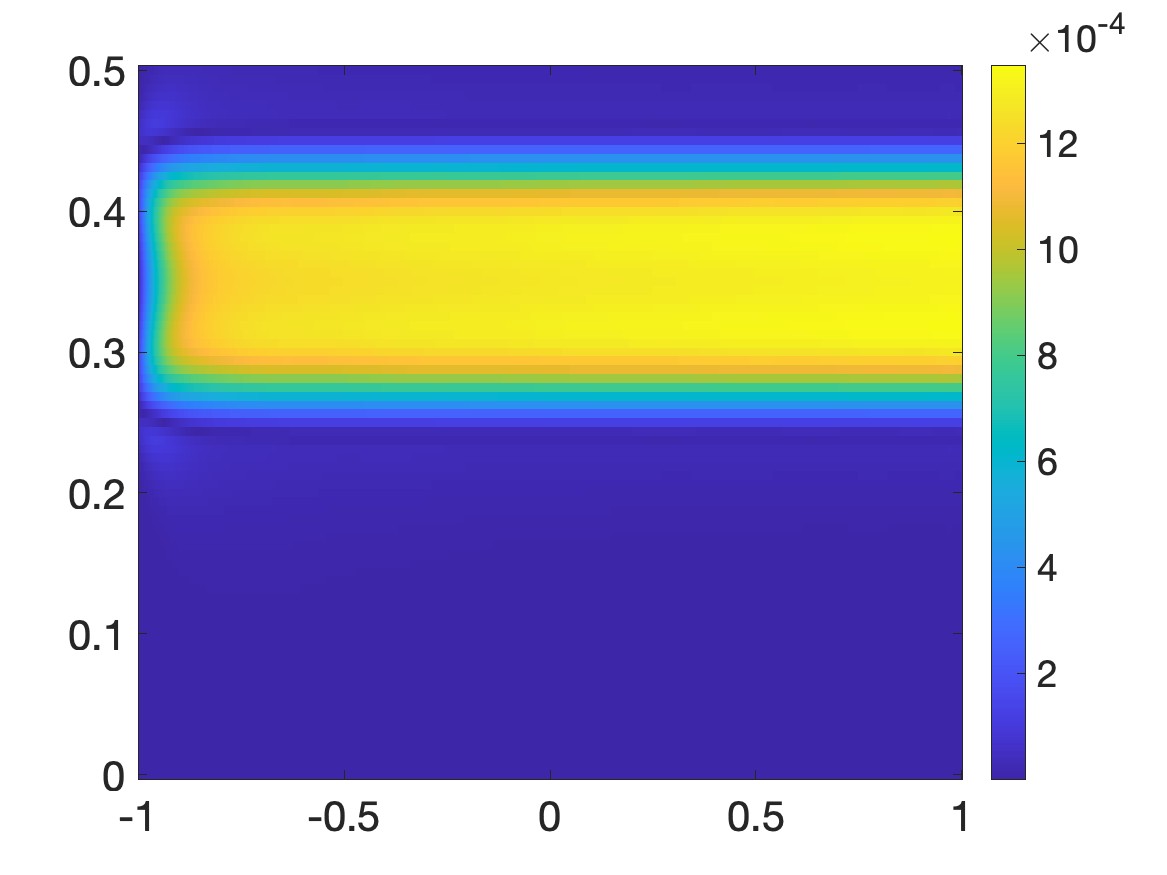}}
	\subfloat[$\varphi_{(15, 10)}$]{\includegraphics[width=.3\textwidth]{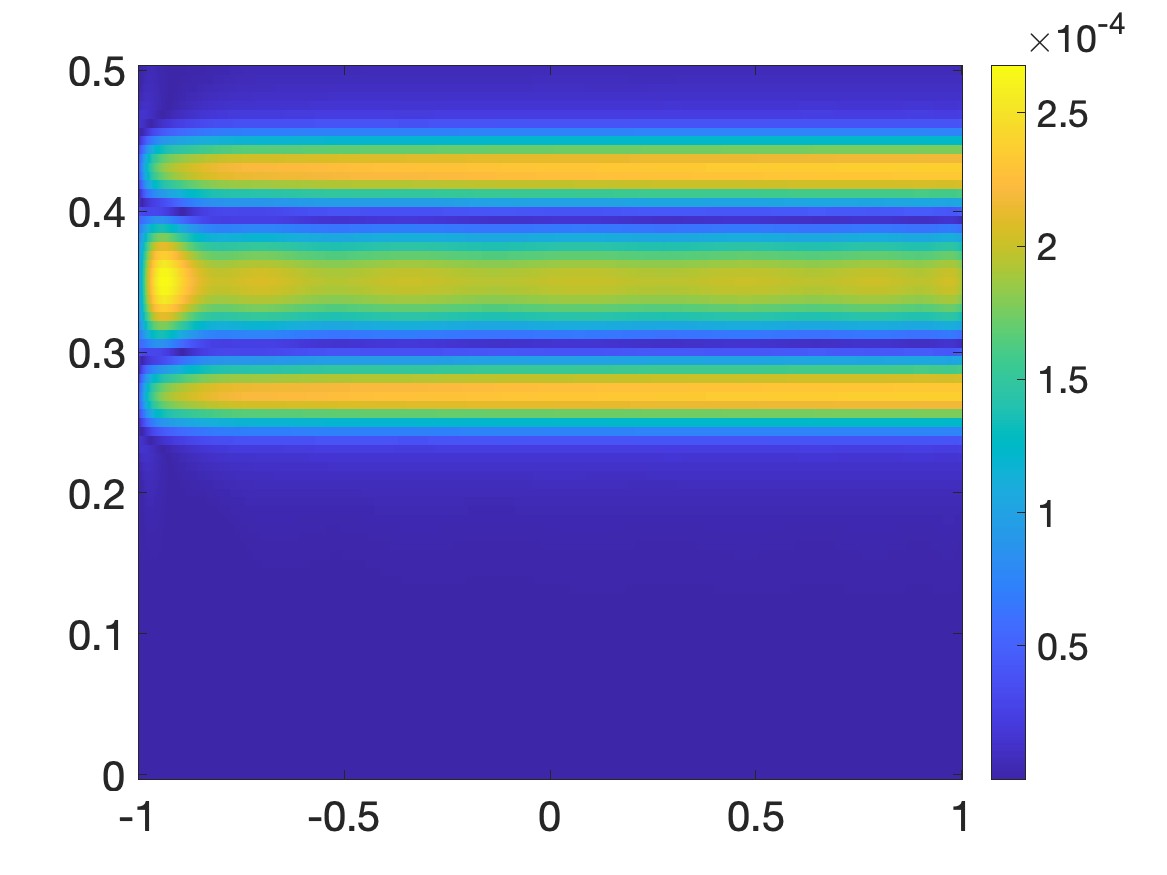}} 
	\caption{\label{ChooseN} The difference $\varphi_{(N_1, N_t)}$ of the data $u(x, -R, t)$ and its truncated Fourier approximation $\sum_{n_1 = 1}^{N_1} \sum_{n_t = 1}^{N_t} u_{(n_1, n_t)}(-R) \bP_{(n_1, n_t)}(x, t)$,  for some values of $(N_1, N_t)$. The data is taken in Test 1 below. Evidently, for Step \ref{s_chooseN} in Algorithm \ref{alg} during the calculation of the desired coefficient in Test 1, we can choose $N_1 = 15$ and $N_t = 10.$}
	\end{center}
\end{figure}

{\bf Step \ref{s1} of Algorithm \ref{alg}.}  In Step \ref{s1}, we select the parameters through a trial-error procedure involving experimentation and adjustments. To do so, we begin with a reference test where the accurate solution is known. Using this reference, we adjust the parameters until Algorithm \ref{alg} produces a satisfactory numerical outcome with data free of noise, i.e., $\delta = 0$. These chosen parameters are then applied across all subsequent tests and various noise levels $\delta$. The reference for our adjustments is Test 1, mentioned below. In all our numerical analyses, these values are set as $\epsilon = 10^{-6.5}$, $\kappa_0 = 10^{-3}$, $z_0 = -10$, and $\lambda = 10$. 

{\bf Step \ref{s2} of Algorithm \ref{alg}.} We need to choose a function $\bv^{(0)}in H$ . A convenient method to compute this function is solving the linear problem obtained by excluding the nonlinearity $\mathcal{F}$. 
More precisely, we solve the following problem 
\begin{equation}
	\left\{
		\begin{array}{ll}
			{\bv^{(0)}}''(z) + \mathcal{S}::\bv^{(0)}(z)  = 0 &z \in (-R, R),\\
			\bv^{(0)}(z) = \mathcal{P}(z) &z = \pm R,\\
			{\bv^{(0)}}'(z) = \mathcal{Q}(z) &z = \pm R
		\end{array}
	\right.
	\label{5.7}
\end{equation}
for the function $\bv^{(0)}.$
We can compute the numerical solution to \eqref{5.7} by the quasi-reversibility method.
We do not present the numerical implementation to solve linear PDEs using the Carleman quasi-reversibility method in this paper. The reader can find the details about this in \cite{LeNguyen:jiip2022, Nguyen:CAMWA2020, Nguyens:jiip2020}.

{\bf Step \ref{step update} of Algorithm \ref{alg}.}
In Step \ref{step update}, we minimize $J_{\bv^{(k)}}$ and set $\bv^{(k+1)}$ as its minimizer. The obtained minimizer $\bv^{(k+1)}$ is actually the solution to 
\begin{equation}
	\left\{
		\begin{array}{ll}
			{\bv^{(k+1)}}''(z) + \mathcal{S}::\bv^{(k+1)}(z) + \mathcal{F}(\bv^{(k)}(z)) = 0 &z \in (-R, R),\\
			\bv^{(k+1)}(z) = \mathcal{P}(z) &z = \pm R,\\
			{\bv^{(k+1)}}'(z) = \mathcal{Q}(z) &z = \pm R.
		\end{array}
	\right.
	\label{5.8}
\end{equation}
The details in implementation to compute the regularized solution $\bv^{(k+1)}$ to \eqref{5.8} were presented in \cite{LeNguyen:jiip2022, Nguyen:CAMWA2020, Nguyens:jiip2020}, in which we employ the optimization package already built in Matlab. We do not repeat it here.

{\bf Steps \ref{if}--\ref{s13}.} Executing these steps is direct and not complicated. For brevity, we do not present them here.

\subsection{Numerical examples}

We show in this section three (3) numerical examples computed by Algorithm \ref{alg}. 

{\bf Test 1.}
In test 1, the true coefficient $c$ is given by
\[
	c^{\rm true}(x, z) = \left\{
		\begin{array}{ll}
			e^{\frac{0.35x^2 + (z - 0.4)^2}{0.55^2 - (0.35x^2 + (z - 0.4)^2)}} & 0.35x^2 + (z - 0.4)^2 < 0.55^2,
			\\
			0 &\mbox{otherwise}.
		\end{array}
	\right.
\] for $(x, z) \in \Omega.$
It is characterized by an ``ellipse" inclusion.
The true and computed coefficient $c$ are displayed in Figure \ref{figtest1}.

\begin{figure}[!ht]
\centering
	\subfloat[The true coefficient $c^{\rm true}$]{\includegraphics[width=.3\textwidth]{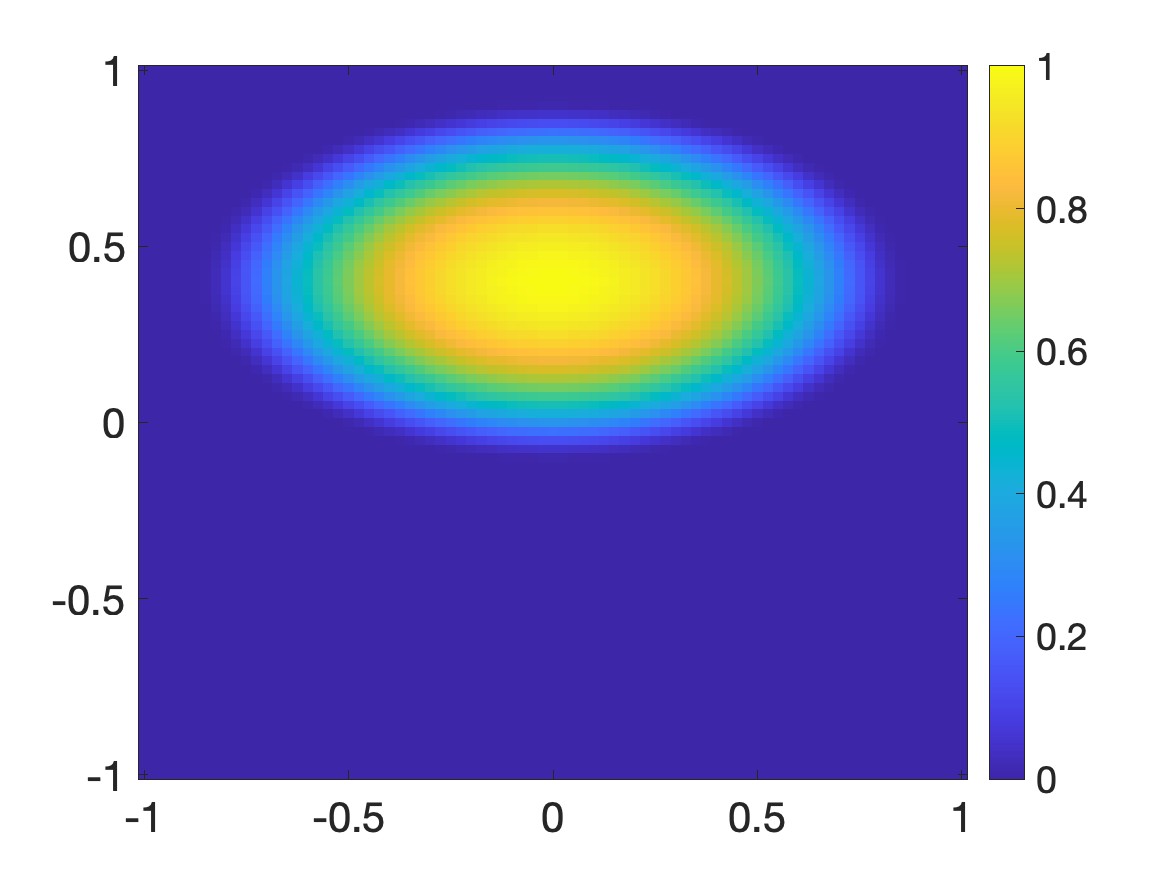}}
	\quad
	\subfloat[The computed coefficient $c^{\rm comp}$]{\includegraphics[width=.3\textwidth]{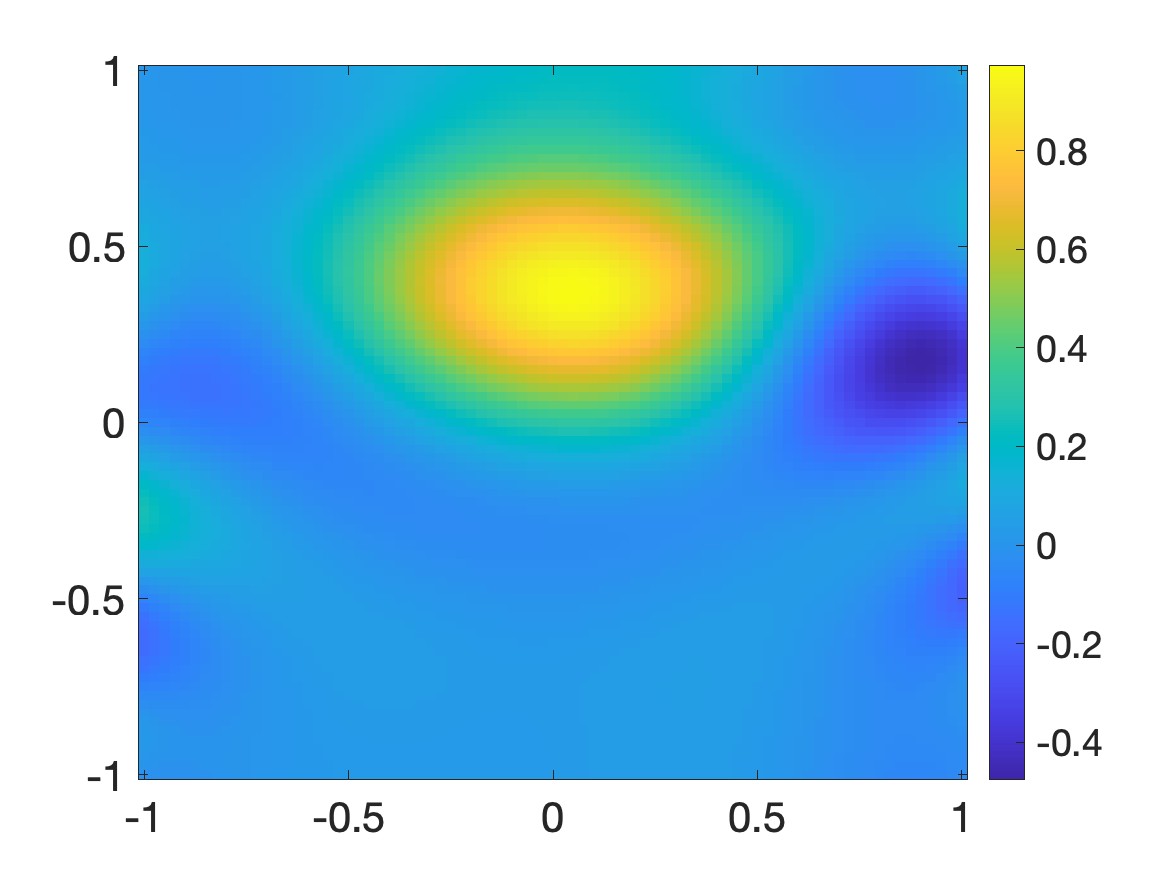}}
	\quad
	\subfloat[\label{fig1c}The consecutive relative error $\ds\frac{\|\bv^{(k )} - \bv^{(k-1)}\|_{L^{\infty}}}{\|\bv^{(k)}\|_{L^\infty}}$.]{\includegraphics[width=.3\textwidth]{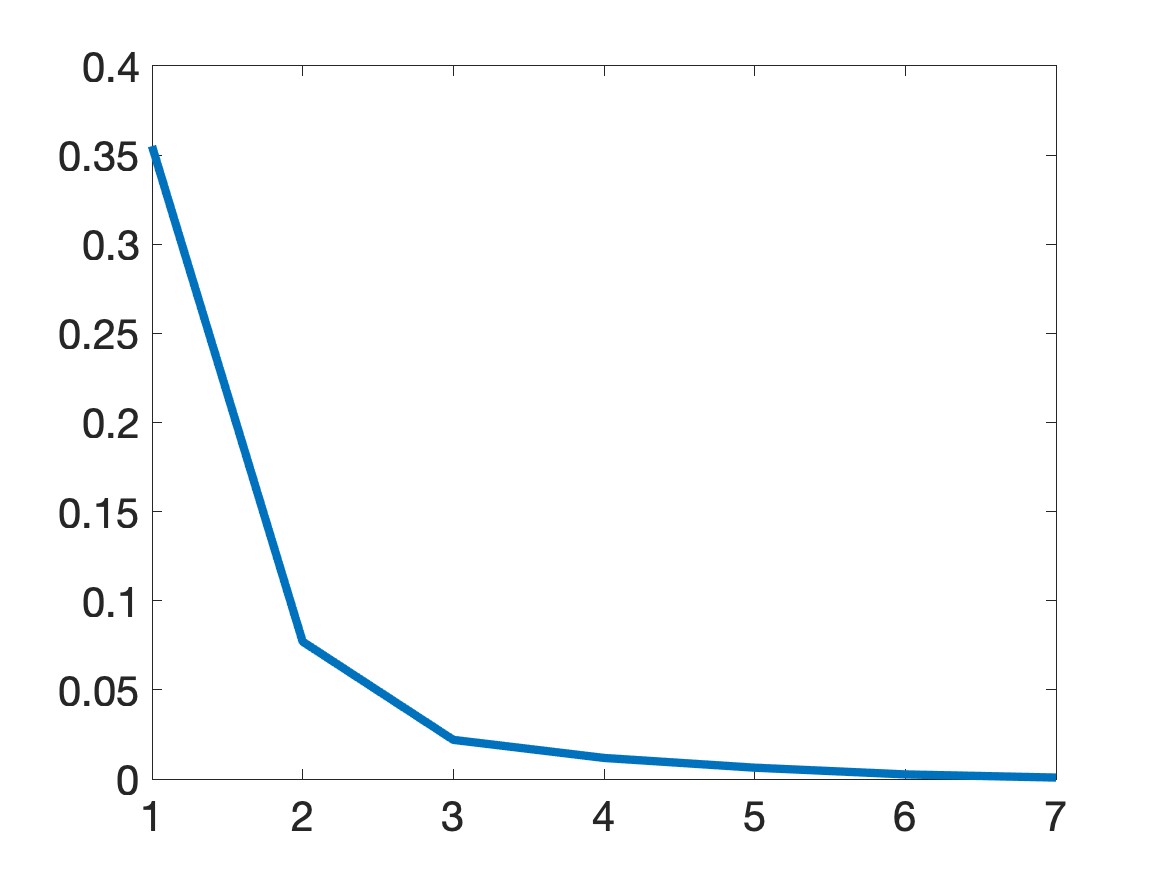}}
\caption{\label{figtest1} (a) The true coefficient, (b) the computed one from data on $\Gamma$ corrupted with $5\%$ of noise, and (c) the relative difference of the reconstructed of the solution $\bv$ to \eqref{2.12} at the $k^{\rm th}$ iteration. Although the data are missing in $\partial \Omega \setminus \Gamma$, the reconstruction is acceptable. The convergence of our method is numerically confirmed.}
\end{figure}

It is evident that Algorithm \ref{alg} generates a satisfactory numerical solution. It is evident that the ``ellipse inclusion" was successfully detected. 
The maximum value of the function $c$ inside the inclusion is $1$. The constructed value is $0.952$ (relative error 4.8\%).
Due to Figure \ref{fig1c}, the stopping criterium of Algorithm \ref{alg} is met after only seven iterations.

{\bf Test 2.} We test the case when the true coefficient is characterized by two horizontal inclusions.
More precisely,
we set
\[
	c^{\rm true}(x, z) =
	\left\{
		\begin{array}{ll}
			1 &\max\{0.25 |x|, 4|z - 0.6|\} < 0.8,\\
			1 &\max\{0.25 |x|, 4|z + 0.6|\} < 0.8,\\
			0 &\mbox{otherwise},
		\end{array}
	\right.
\quad
	\mbox{for all } (x, z) \in \Omega.	
\]
The true and computed solutions to Problem \ref{p} are shown in Figure \ref{figtest2}
\begin{figure}[!ht]
\centering
	\subfloat[The true coefficient $c^{\rm true}$]{\includegraphics[width=.3\textwidth]{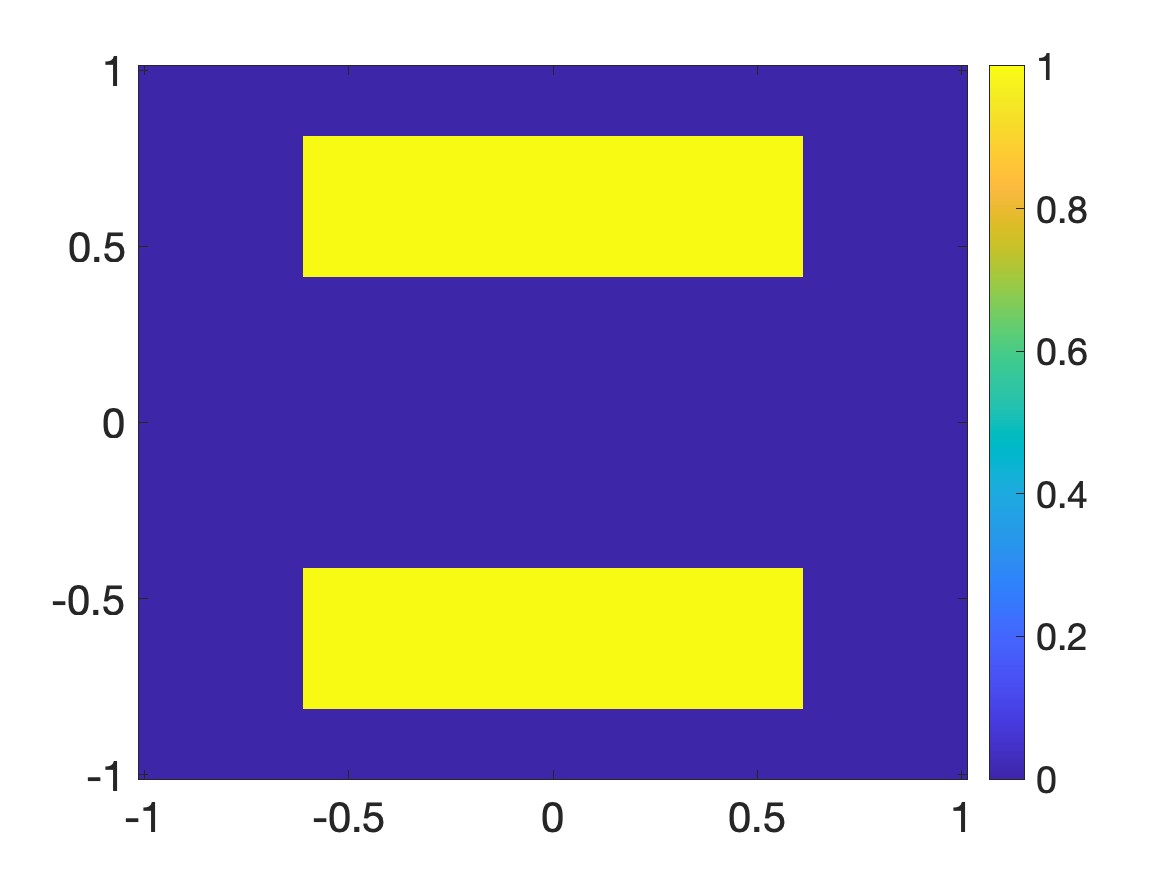}}
	\quad
	\subfloat[The computed coefficient $c^{\rm comp}$]{\includegraphics[width=.3\textwidth]{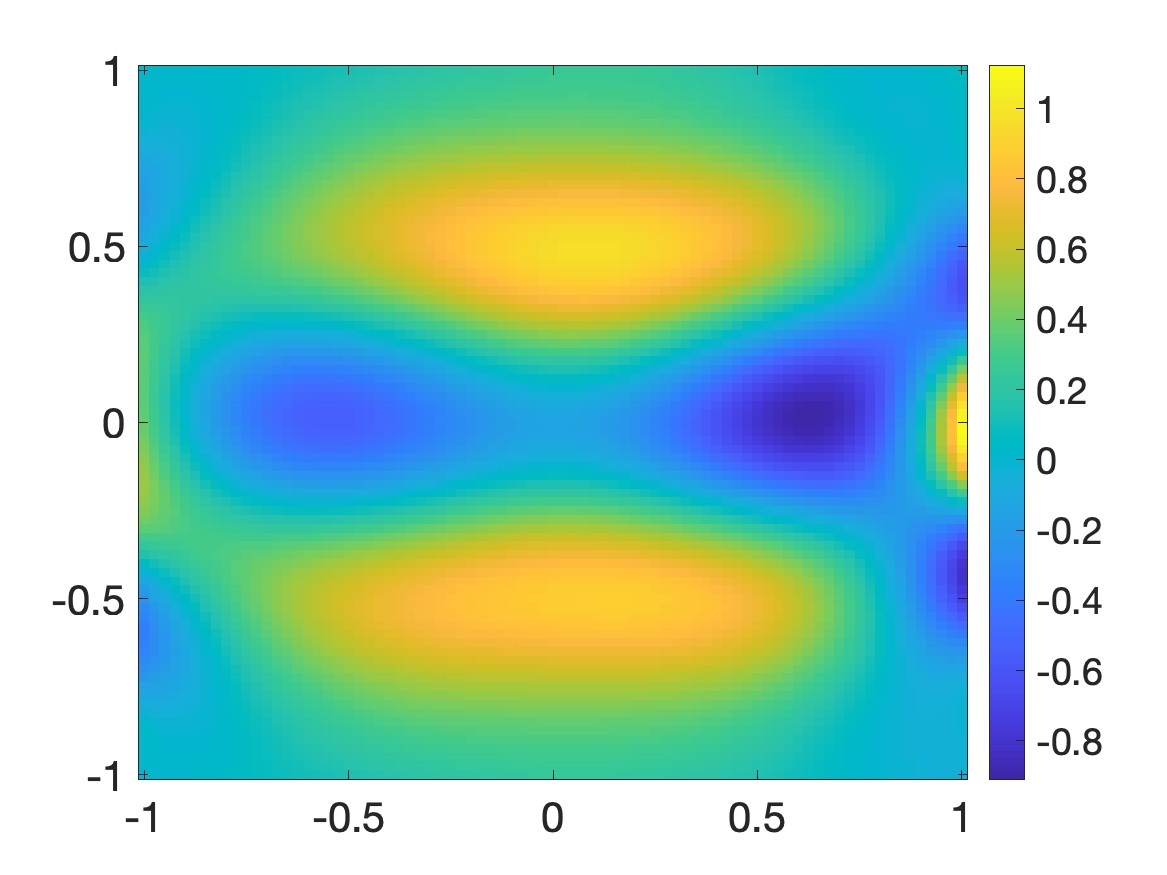}}
	\quad
	\subfloat[\label{fig2c}The consecutive relative error $\ds\frac{\|\bv^{(k )} - \bv^{(k-1)}\|_{L^{\infty}}}{\|\bv^{(k)}\|_{L^\infty}}$.]{\includegraphics[width=.3\textwidth]{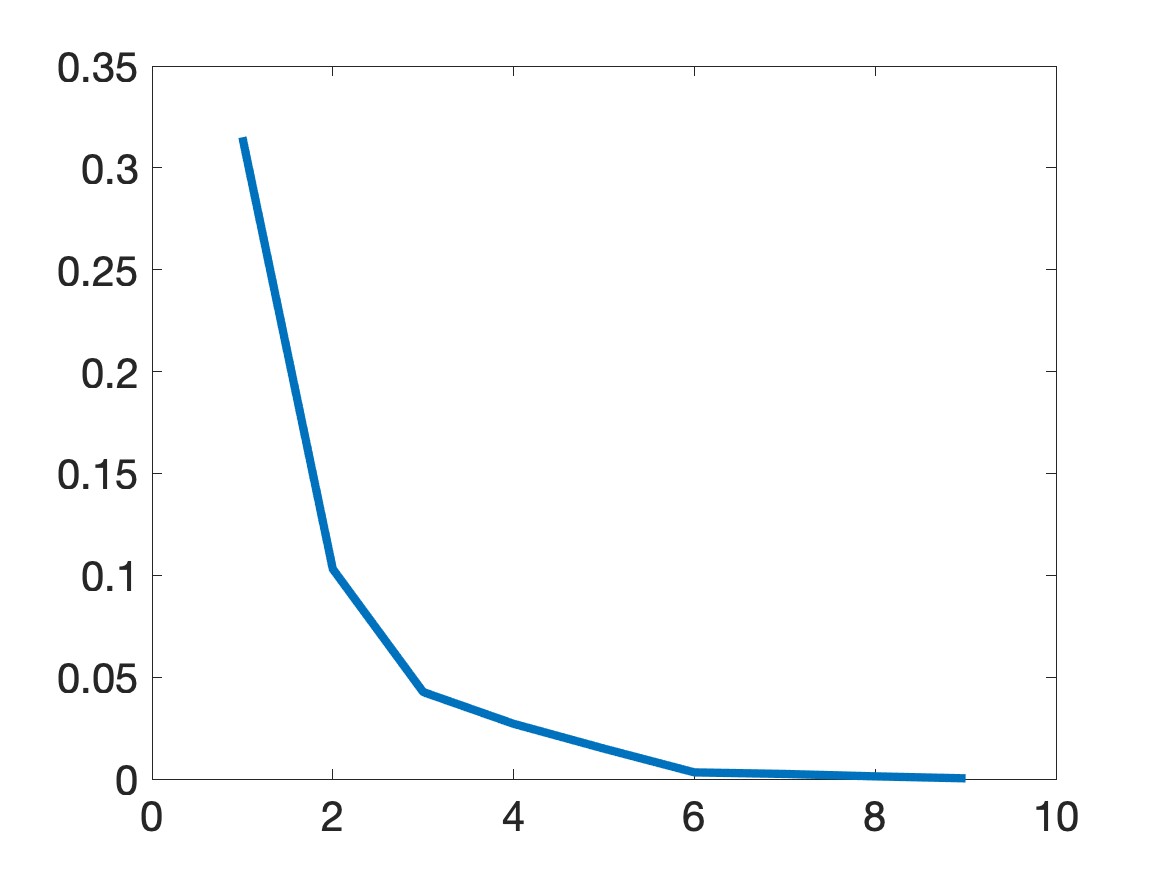}}
\caption{\label{figtest2} (a) The true coefficient, (b) the computed one from data on $\Gamma$ corrupted with $5\%$ of noise, and (c) the relative difference of the reconstructed of the solution $\bv$ to \eqref{2.12} at the $k^{\rm th}$ iteration. "Despite the absence of data in $\partial \Omega \setminus \Gamma$, the reconstruction remains satisfactory. Our method's fast convergence has been verified numerically.}
\end{figure}

As in Test 1, we can see that Algorithm \ref{alg} provides a satisfactory numerical solution. It is evident that both horizontal inclusions were successfully identified. 
The maximum value of the function $c$ inside each inclusion is $1$. The constructed value of $c$ inside the upper inclusion is $0.963$ (relative error 3.7\%), and the one inside the lower inclusion is $0.875$ (relative error 12.4\%).
Due to Figure \ref{fig2c}, Algorithm \ref{alg} stops at only nine iterations.

{\bf Test 3.} 
Test 3 checks the case when the true coefficient $c$ has a $T$ inclusion. That means the function $c^{\rm true}$ takes the value 1 inside a letter $T$ and $0$ otherwise. We refer the reader to Figure \ref{figtest3} for the image of the true and computed coefficients.

\begin{figure}[!ht]
\centering
	\subfloat[The true coefficient $c^{\rm true}$]{\includegraphics[width=.3\textwidth]{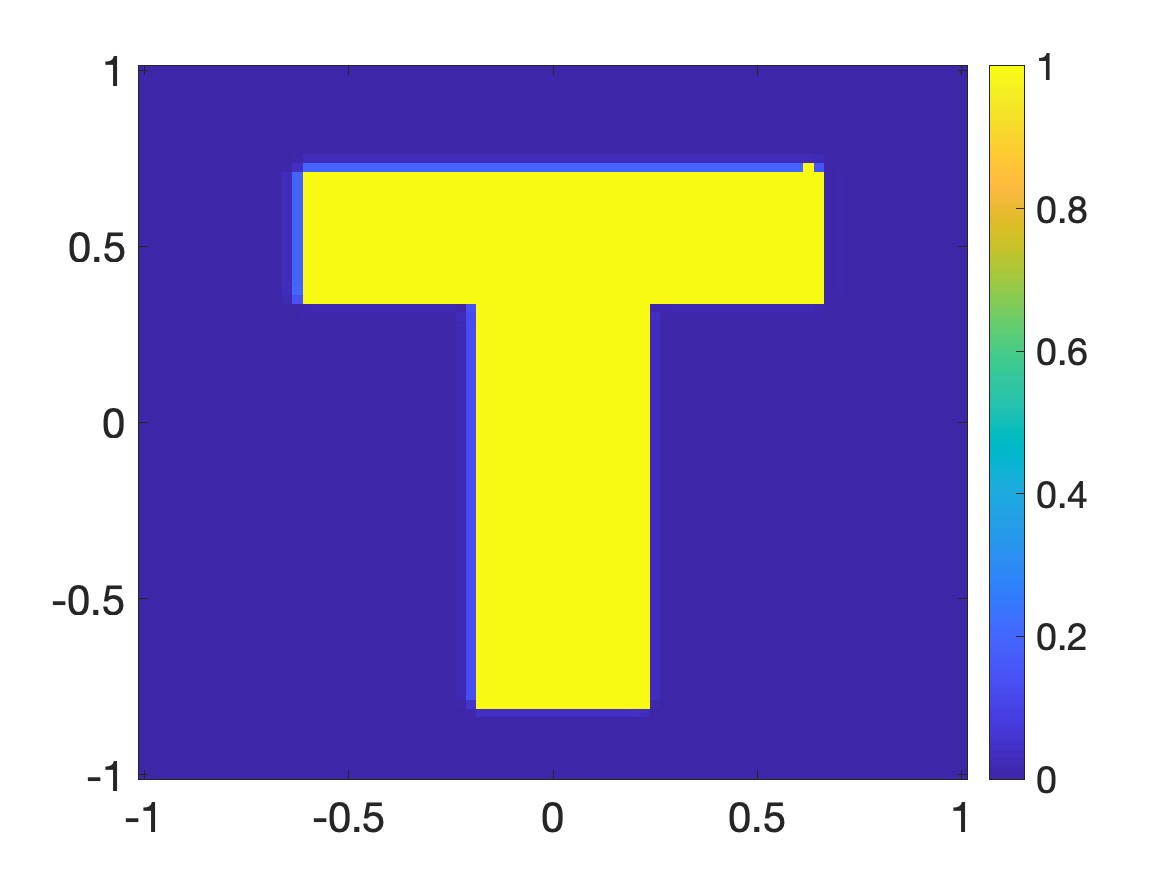}}
	\quad
	\subfloat[The computed coefficient $c^{\rm comp}$]{\includegraphics[width=.3\textwidth]{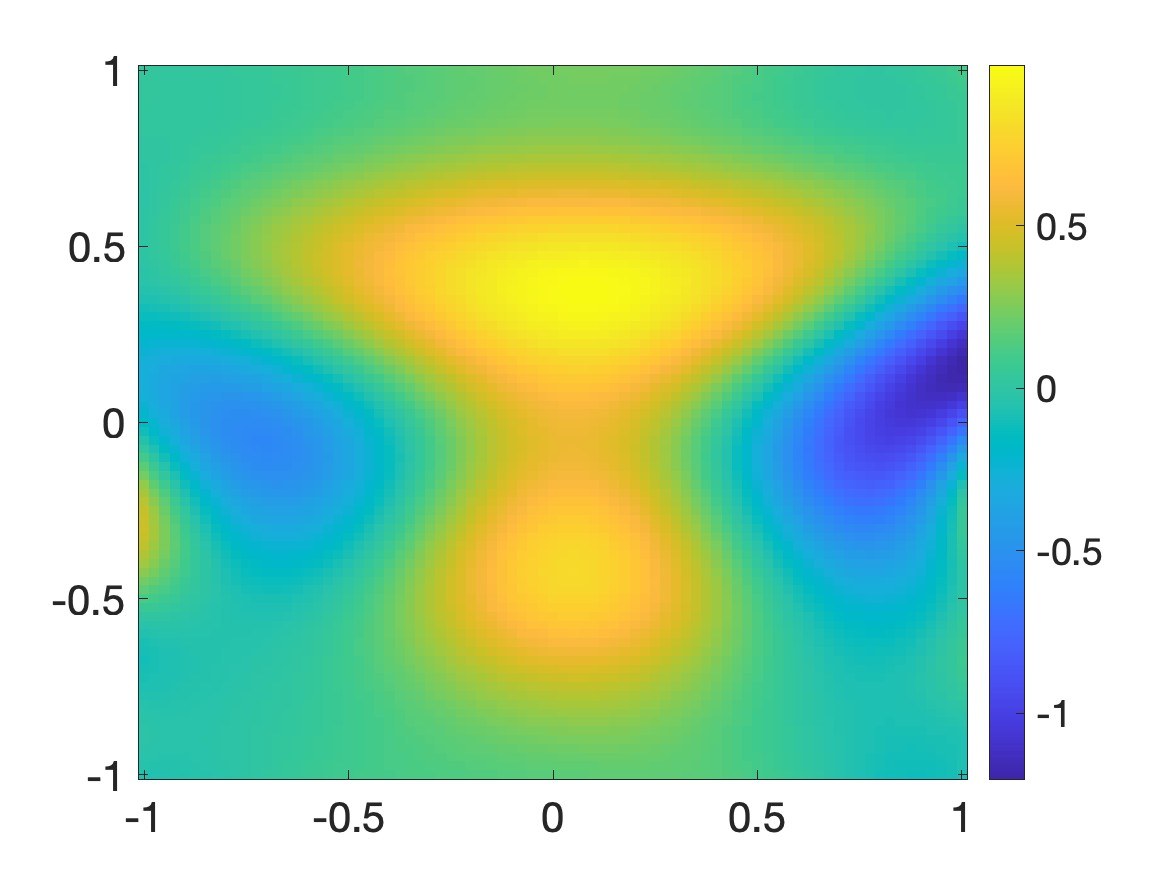}}
	\quad
	\subfloat[\label{fig3c}The consecutive relative error $\ds\frac{\|\bv^{(k )} - \bv^{(k-1)}\|_{L^{\infty}}}{\|\bv^{(k)}\|_{L^\infty}}$.]{\includegraphics[width=.3\textwidth]{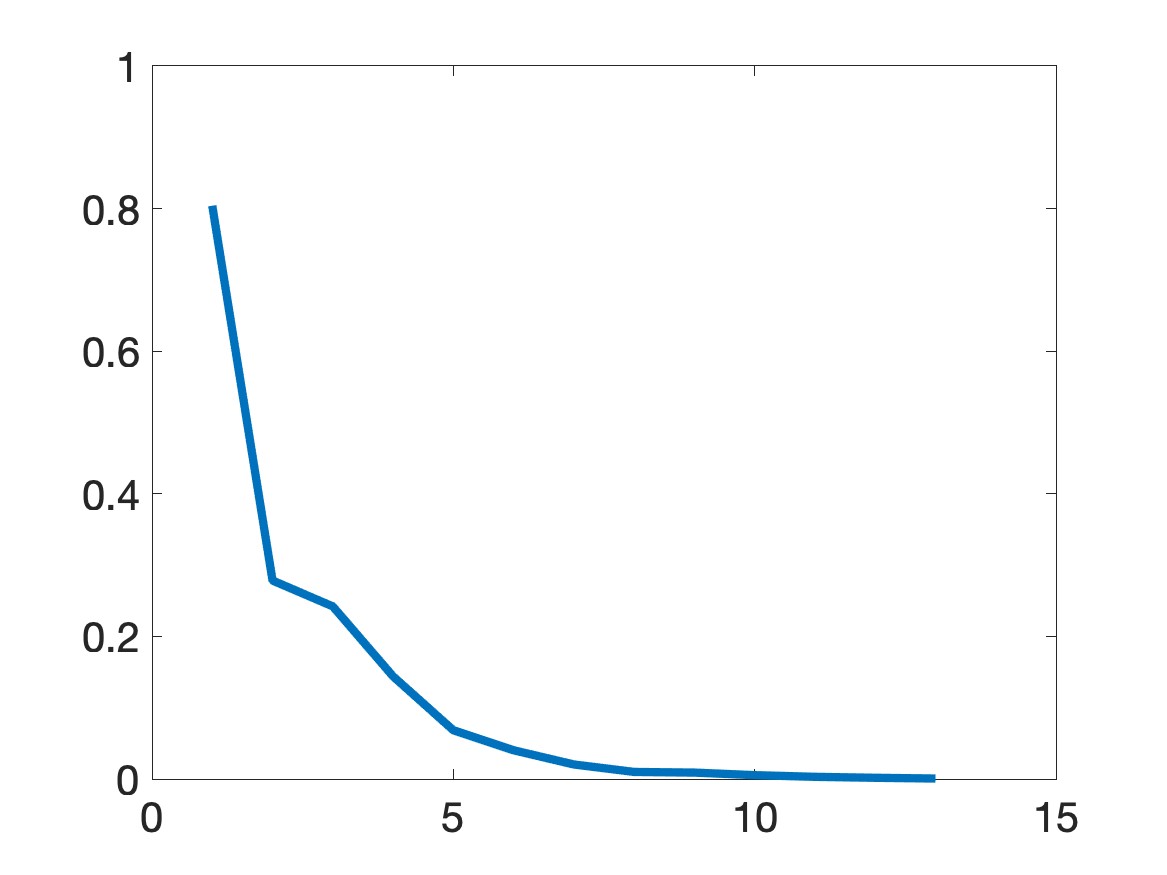}}
\caption{\label{figtest3} (a) The true coefficient, (b) the computed one from data on $\Gamma$ corrupted with $5\%$ of noise, and (c) the relative difference of the reconstructed of the solution $\bv$ to \eqref{2.12} at the $k^{\rm th}$ iteration. "Despite the absence of data in $\partial \Omega \setminus \Gamma$, the reconstruction of the letter $T$ is acceptable. Our method's fast convergence has been verified numerically.}
\end{figure}

As in Test 1 and Test 2, we can see that the $T$ inclusion was successfully found. 
The maximum value of the function $c$ inside each inclusion is $1$. The constructed value of $c$ inside the  inclusion  is $0.98$ (relative error 2\%).
Due to Figure \ref{fig3c}, Algorithm \ref{alg} stops at only 13 iterations.
\begin{Remark}
	The computational cost of Algorithm \ref{alg} is not expensive. 
	In fact, we only need to solve several 1D linear over-determined problems in Steps \ref{if}-\ref{step7}.
	We have used a personal computer iMac with a 3.2GHz Intel Core i5 Processor and memory of 24GB, not a professional workstation, to compute the numerical solutions above. It took 46.93, 56.56, and 77.12 seconds to complete all computational tasks for the inverse problem in Test 1, Test 2, and Test 3, respectively. 
\end{Remark}

\begin{Remark}
	It is worth mentioning that the images of the inclusions in the tests above are not perfect. One can find some artifacts occurring in the numerical results.
	However, these errors are acceptable because we solve the inverse problem when the data are given only on the part $\Gamma$ of $\partial \Omega.$ 
\end{Remark}

\section{Concluding remarks}\label{sec6}

In this study, we provide a numerical method to solve a nonlinear coefficient inverse problem for parabolic equations.
This inverse problem has numerous practical applications. To obtain these solutions, we employ the polynomial-differential basis, converting the inverse problem into a set of 1D nonlinear equations. Next, we introduce a method to address this nonlinear system. Our method is based on a combination of the Picard iteration, the quasi-reversibility method, and a Carleman estimate. We provide analytic proof of the method's convergence and demonstrate its effectiveness through some numerical examples.

 \section*{Acknowledgement}
The works of RA, TTL, and LHN were partially supported  by National Science Foundation grant DMS-2208159, by funds provided by the Faculty Research Grant program at UNC Charlotte Fund No. 111272,  and by
the CLAS small grant provided by the College of Liberal Arts \& Sciences, UNC Charlotte. 
The works of CP were supported in part by National Science Foundation grant DMS-2150179.

\bibliographystyle{plain}
\bibliography{../../../../mybib}

\begin{thebibliography}{10}

\bibitem{AbhishekLeNguyenKhan}
A.~Abhishek, T.~T. Le, L.~H. Nguyen, and T.~Khan.
\newblock The {C}arleman-{N}ewton method to globally reconstruct a source term
  for nonlinear parabolic equation.
\newblock {\em preprint, arXiv:2209.08011}, 2022.

\bibitem{KlibanovNik:ra2017}
A.~B. Bakushinskii, M.~V. Klibanov, and N.~A. Koshev.
\newblock Carleman weight functions for a globally convergent numerical method
  for ill-posed {C}auchy problems for some quasilinear {PDEs}.
\newblock {\em Nonlinear Anal. Real World Appl.}, 34:201--224, 2017.

\bibitem{Becacheelal:AIMS2015}
E.~B\'ecache, L.~Bourgeois, L.~Franceschini, and J.~Dard\'e.
\newblock Application of mixed formulations of quasi-reversibility to solve
  ill-posed problems for heat and wave equations: The 1d case.
\newblock {\em Inverse Problems \& Imaging}, 9(4):971--1002, 2015.

\bibitem{BeilinaKlibanovBook}
L.~Beilina and M.~V. Klibanov.
\newblock {\em Approximate Global Convergence and Adaptivity for Coefficient
  Inverse Problems}.
\newblock Springer, New York, 2012.

\bibitem{Borceaetal:ip2014}
L.~Borcea, V.~Druskin, A.~V. Mamonov, and M.~Zaslavsky.
\newblock A model reduction approach to numerical inversion for a parabolic
  partial differential equation.
\newblock {\em Inverse Problems}, 30:125011, 2014.

\bibitem{Bourgeois:ip2006}
L.~Bourgeois.
\newblock Convergence rates for the quasi-reversibility method to solve the
  {C}auchy problem for {L}aplace's equation.
\newblock {\em Inverse Problems}, 22:413--430, 2006.

\bibitem{BourgeoisDarde:ip2010}
L.~Bourgeois and J.~Dard\'e.
\newblock A duality-based method of quasi-reversibility to solve the {C}auchy
  problem in the presence of noisy data.
\newblock {\em Inverse Problems}, 26:095016, 2010.

\bibitem{BourgeoisPonomarevDarde:ipi2019}
L.~Bourgeois, D.~Ponomarev, and J.~Dard\'e.
\newblock An inverse obstacle problem for the wave equation in a finite time
  domain.
\newblock {\em Inverse Probl. Imaging}, 13(2):377--400, 2019.

\bibitem{BukhgeimKlibanov:smd1981}
A.~L. Bukhgeim and M.~V. Klibanov.
\newblock Uniqueness in the large of a class of multidimensional inverse
  problems.
\newblock {\em Soviet Math. Doklady}, 17:244--247, 1981.

\bibitem{CaoLesnic:nmpde2018}
K.~Cao and D.~Lesnic.
\newblock Determination of space-dependent coefficients from temperature
  measurements using the conjugate gradient method.
\newblock {\em Numer Methods Partial Differential Eq.}, 34:1370--1400, 2018.

\bibitem{CaoLesnic:amm2019}
K.~Cao and D.~Lesnic.
\newblock Simultaneous reconstruction of the perfusion coefficient and initial
  temperature from time-average integral temperature measurements.
\newblock {\em Applied Mathematical Modelling}, 68:523--539, 2019.

\bibitem{ClasonKlibanov:sjsc2007}
C.~Clason and M.~V. Klibanov.
\newblock The quasi-reversibility method for thermoacoustic tomography in a
  heterogeneous medium.
\newblock {\em SIAM J. Sci. Comput.}, 30:1--23, 2007.

\bibitem{Dadre:ipi2016}
J.~Dard\'e.
\newblock Iterated quasi-reversibility method applied to elliptic and parabolic
  data completion problems.
\newblock {\em Inverse Problems and Imaging}, 10:379--407, 2016.

\bibitem{Isakov:ip1999}
V.~Isakov.
\newblock Some inverse problems for the diffusion equation.
\newblock {\em Inverse Problems}, 15(1):3--10, 1999.

\bibitem{KaltenbacherRundell:ipi2019}
B.~Kaltenbacher and W.~Rundell.
\newblock Regularization of a backwards parabolic equation by fractional
  operators.
\newblock {\em Inverse Probl. Imaging}, 13(2):401--430, 2019.

\bibitem{KeungZou:ip1998}
Y.~L. Keung and J.~Zou.
\newblock Numerical identifications of parameters in parabolic systems.
\newblock {\em Inverse Problems}, 14:83--100, 1998.

\bibitem{VoKlibanovNguyen:IP2020}
V.~A. Khoa, G.~W. Bidney, M.~V. Klibanov, L.~H. Nguyen, L.~Nguyen, A.~Sullivan,
  and V.~N. Astratov.
\newblock Convexification and experimental data for a {3D} inverse scattering
  problem with the moving point source.
\newblock {\em Inverse Problems}, 36:085007, 2020.

\bibitem{Khoaelal:IPSE2021}
V.~A. Khoa, G.~W. Bidney, M.~V. Klibanov, L.~H. Nguyen, L.~Nguyen, A.~Sullivan,
  and V.~N. Astratov.
\newblock An inverse problem of a simultaneous reconstruction of the dielectric
  constant and conductivity from experimental backscattering data.
\newblock {\em Inverse Problems in Science and Engineering}, 29(5):712--735,
  2021.

\bibitem{KhoaKlibanovLoc:SIAMImaging2020}
V.~A. Khoa, M.~V. Klibanov, and L.~H. Nguyen.
\newblock Convexification for a 3{D} inverse scattering problem with the moving
  point source.
\newblock {\em SIAM J. Imaging Sci.}, 13(2):871--904, 2020.

\bibitem{Klibanov:jiipp2013}
M.~V. Klibanov.
\newblock Carleman estimates for global uniqueness, stability and numerical
  methods for coefficient inverse problems.
\newblock {\em J. Inverse and Ill-Posed Problems}, 21:477--560, 2013.

\bibitem{Klibanov:anm2015}
M.~V. Klibanov.
\newblock Carleman estimates for the regularization of ill-posed {C}auchy
  problems.
\newblock {\em Applied Numerical Mathematics}, 94:46--74, 2015.

\bibitem{Klibanov:ip2015}
M.~V. Klibanov.
\newblock Carleman weight functions for solving ill-posed {C}auchy problems for
  quasilinear {PDE}s.
\newblock {\em Inverse Problems}, 31:125007, 2015.

\bibitem{Klibanov:jiip2017}
M.~V. Klibanov.
\newblock Convexification of restricted {D}irichlet to {N}eumann map.
\newblock {\em J. Inverse and Ill-Posed Problems}, 25(5):669--685, 2017.

\bibitem{KlibanovIoussoupova:SMA1995}
M.~V. Klibanov and O.~V. Ioussoupova.
\newblock Uniform strict convexity of a cost functional for three-dimensional
  inverse scattering problem.
\newblock {\em SIAM J. Math. Anal.}, 26:147--179, 1995.

\bibitem{KlibanovNguyenTran:JCP2022}
M.~V. Klibanov, L.~H. Nguyen, and H.~V. Tran.
\newblock Numerical viscosity solutions to {H}amilton-{J}acobi equations via a
  {C}arleman estimate and the convexification method.
\newblock {\em Journal of Computational Physics}, 451:110828, 2022.

\bibitem{KlibanovSantosa:SIAMJAM1991}
M.~V. Klibanov and F.~Santosa.
\newblock A computational quasi-reversibility method for {C}auchy problems for
  {L}aplace's equation.
\newblock {\em SIAM J. Appl. Math.}, 51:1653--1675, 1991.

\bibitem{klibanovYagola:arxiv2019}
M.~V. Klibanov and A.~G. Yagola.
\newblock Convergent numerical methods for parabolic equations with reversed
  time via a new {C}arleman estimate.
\newblock {\em preprint}, 2019.

\bibitem{LattesLions:e1969}
R.~Latt\`es and J.~L. Lions.
\newblock {\em The Method of Quasireversibility: Applications to Partial
  Differential Equations}.
\newblock Elsevier, New York, 1969.

\bibitem{LeLeNguyen:Arxiv2022}
P.~N.~H. Le, T.~T. Le, and L.~H. Nguyen.
\newblock The {C}arleman convexification method for {H}amilton-{J}acobi
  equations on the whole space.
\newblock {\em preprint, arXiv:2206.09824}, 2022.

\bibitem{LeCON2023}
T.~T. Le.
\newblock Global reconstruction of initial conditions of nonlinear parabolic
  equations via the {C}arleman-contraction method.
\newblock In D-L. Nguyen, L.~H. Nguyen, and T-P. Nguyen, editors, {\em Advances
  in Inverse problems for Partial Differential Equations}, volume 784 of {\em
  Contemporary Mathematics}, pages 23--42. American Mathematical Society, 2023.

\bibitem{ThuyKhoaKlibanovLocBidneyAstratov:2023}
T.~T. Le, V.~A. Khoa, M.~V. Klibanov, L.~H. Nguyen, G.~W. Bidney, and V.~N.
  Astratov.
\newblock Numerical verification of the convexification method for a
  frequency-dependent inverse scattering problem with experimental data.
\newblock {\em to appear in Journal of Applied and Industrial Mathematics,
  preprint arXiv:2306.00761}, 2023.

\bibitem{LeNguyen:jiip2022}
T.~T. Le and L.~H. Nguyen.
\newblock A convergent numerical method to recover the initial condition of
  nonlinear parabolic equations from lateral {C}auchy data.
\newblock {\em Journal of Inverse and Ill-posed Problems,}, 30(2):265--286,
  2022.

\bibitem{LeNguyen:JSC2022}
T.~T. Le and L.~H. Nguyen.
\newblock The gradient descent method for the convexification to solve boundary
  value problems of quasi-linear {PDEs} and a coefficient inverse problem.
\newblock {\em Journal of Scientific Computing}, 91(3):74, 2022.

\bibitem{LeNguyenNguyenPowell:JOSC2021}
T.~T. Le, L.~H. Nguyen, T-P. Nguyen, and W.~Powell.
\newblock The quasi-reversibility method to numerically solve an inverse source
  problem for hyperbolic equations.
\newblock {\em Journal of Scientific Computing}, 87:90, 2021.

\bibitem{LeNguyenTran:CAMWA2022}
T.~T. Le, L.~H. Nguyen, and H.~V. Tran.
\newblock A {C}arleman-based numerical method for quasilinear elliptic
  equations with over-determined boundary data and applications.
\newblock {\em Computers and Mathematics with Applications}, 125:13--24, 2022.

\bibitem{LiNguyen:IPSE2020}
Q.~Li and L.~H. Nguyen.
\newblock Recovering the initial condition of parabolic equations from lateral
  {C}auchy data via the quasi-reversibility method.
\newblock {\em Inverse Problems in Science and Engineering}, 28:580--598, 2020.

\bibitem{TuanKhoaAu:SIAM2019}
H.~T. Nguyen, V.~A. Khoa, and V.~A. Vo.
\newblock Analysis of a quasi-reversibility method for a terminal value
  quasi-linear parabolic problem with measurements.
\newblock {\em SIAM Journal on Mathematical Analysis}, 51:60--85, 2019.

\bibitem{LocNguyen:ip2019}
L.~H. Nguyen.
\newblock An inverse space-dependent source problem for hyperbolic equations
  and the {L}ipschitz-like convergence of the quasi-reversibility method.
\newblock {\em Inverse Problems}, 35:035007, 2019.

\bibitem{Nguyen:CAMWA2020}
L.~H. Nguyen.
\newblock A new algorithm to determine the creation or depletion term of
  parabolic equations from boundary measurements.
\newblock {\em Computers and Mathematics with Applications}, 80:2135--2149,
  2020.

\bibitem{Nguyen:AVM2023}
L.~H. Nguyen.
\newblock The {C}arleman contraction mapping method for quasilinear elliptic
  equations with over-determined boundary data.
\newblock {\em Acta Mathematica Vietnamica, DOI:
  https://doi.org/10.1007/s40306-023-00500-w}, 2023.

\bibitem{NguyenKlibanov:ip2022}
L.~H. Nguyen and M.~V. Klibanov.
\newblock Carleman estimates and the contraction principle for an inverse
  source problem for nonlinear hyperbolic equations.
\newblock {\em Inverse Problems}, 38:035009, 2022.

\bibitem{NguyenLiKlibanov:2019}
L.~H. Nguyen, Q.~Li, and M.~V. Klibanov.
\newblock A convergent numerical method for a multi-frequency inverse source
  problem in inhomogenous media.
\newblock {\em Inverse Problems and Imaging}, 13:1067--1094, 2019.

\bibitem{NguyenLeNguyenKlibanov:2023}
P.~M. Nguyen, T.~T. Le, L.~H. Nguyen, and M.~V. Klibanov.
\newblock Numerical differentiation by the polynomial-exponential basis.
\newblock {\em to appear in Journal of Applied and Industrial Mathematics,
  preprint arXiv:2304.05909}, 2023.

\bibitem{Nguyens:jiip2020}
P.~M. Nguyen and L.~H. Nguyen.
\newblock A numerical method for an inverse source problem for parabolic
  equations and its application to a coefficient inverse problem.
\newblock {\em Journal of Inverse and Ill-posed Problems}, 38:232--339, 2020.

\bibitem{PrilepkoKostin:RASBM1993}
A.~I. Prilepko and A.~B. Kostin.
\newblock On certain inverse problems for parabolic equations with final and
  integral observation.
\newblock {\em Russ. Acad. Sci. Sb. Math.}, 75:473--490, 1993.

\bibitem{Prilepko:pam2000}
A.~I. Prilepko, D.~G. Orlovsky, and I.~A. Vasin.
\newblock {\em Methods for solving inverse problems in mathematical physics},
  volume 321.
\newblock Pure and Applied Mathematics, Marcel Dekker, New Youk, 2000.

\bibitem{KlibanovAlexeyNguyen:SISC2019}
A.~V. Smirnov, M.~V. Klibanov, and L.~H. Nguyen.
\newblock On an inverse source problem for the full radiative transfer equation
  with incomplete data.
\newblock {\em SIAM Journal on Scientific Computing}, 41:B929--B952, 2019.

\bibitem{Tuan:ip2017}
N.~H. Tuan, V.~V. Au, V.~A. Khoa, and D.~Lesnic.
\newblock Identification of the population density of a species model with
  nonlocal diffusion and nonlinear reaction.
\newblock {\em Inverse Problems}, 33:055019, 2017.

\bibitem{YangYuDeng:amm2008}
L.~Yang, J-N. Yu, and Y-C. Deng.
\newblock An inverse problem of identifying the coefficient of parabolic
  equation.
\newblock {\em Applied Mathematical Modelling}, 32:1984--1995, 2008.

\end{thebibliography}

\end{document}